\documentclass[11pt,a4paper]{article}
\usepackage[utf8]{inputenc}
\usepackage{amsthm,amsmath,amsfonts,accents}
\usepackage{hyperref,thm-restate}
\usepackage{cleveref}
\usepackage[mathlines]{lineno}
\usepackage[textsize=footnotesize]{todonotes}
\usepackage{lmodern}
\usepackage[T1]{fontenc}
\usepackage{typearea}
\usepackage{tikz}
\usetikzlibrary{automata,arrows.meta}

\title{Extremal values of semi-regular continuants
and codings of interval exchange transformations}
\author{Alessandro De Luca\textsuperscript{1}, Marcia Edson\textsuperscript{2},
Luca Q. Zamboni\textsuperscript{3}}
\date{}

\newtheorem{mainthm}{Theorem}
\newtheorem{thm}{Theorem}[section]
\newtheorem{prop}[thm]{Proposition}
\newtheorem{lemma}[thm]{Lemma}
\newtheorem{cor}[thm]{Corollary}
\theoremstyle{definition}
\newtheorem{maindef}[mainthm]{Definition}
\newtheorem{defin}[thm]{Definition}
\newtheorem{rem}[thm]{Remark}
\newtheorem{example}[thm]{Example}
\newcommand{\F}{L}
\newcommand{\A}{\mathbb A}
\newcommand{\eps}{\varepsilon}
\newcommand{\pv}{\mathbf v}
\newcommand{\nats}{\mathbb N}
\newcommand{\ints}{\mathbb Z}
\newcommand{\reals}{\mathbb R}
\newcommand*{\dt}[1]{%
  \accentset{\mbox{\large\bfseries .}}{#1}}
\newcommand{\Gs}{\dt{\mathcal G}}
\newcommand{\G}{\mathcal G}

\newcommand{\Ks}{ \dt {K}}

\DeclareMathOperator{\perm}{perm}
\crefname{enumi}{}{items}

\begin{document}
\maketitle
\date
\begin{center}
\textsuperscript{1}DIETI, Universit\`a di Napoli Federico II,
via Claudio 21, 80125 Napoli, Italy\\
\texttt{alessandro.deluca@unina.it}\\
\textsuperscript{2}
\texttt{marciar.edson@gmail.com}\\
\textsuperscript{3}  Institut  Camille  Jordan, CNRS  UMR  5208,
Universit\'{e} de Lyon, Universit\'{e} Lyon  1,\\
43 boulevard du 11 novembre 1918, F69622 Villeurbanne Cedex, France\\
\texttt{zamboni@math.univ-lyon1.fr}
\end{center}

\begin{abstract}
Given a set $\A$ consisting of positive integers $a_1<a_2<\cdots <a_k$  and a $k$-term partition $P: n_1+n_2 + \cdots + n_k=n,$  find the extremal denominators of the regular and semi-regular continued fraction $[0;x_1,x_2,\ldots,x_n]$ with partial quotients $x_i\in \A$ and where each $a_i$ occurs precisely $n_i$ times in the sequence $x_1,x_2,\ldots,x_n.$ In 1983, G. Ramharter gave an explicit description of the extremal arrangements of the regular continued fraction and the minimizing arrangement for the semi-regular continued fraction and showed that in each case the arrangement is unique up to reversal and independent of the actual values of the positive integers $a_i.$ 
However, an explicit determination of a maximizing arrangement for the semi-regular continuant turned out to be substantially more difficult. Ramharter conjectured that
as in the other three cases, the maximizing arrangement is unique (up to reversal) and depends only on the partition $P$ and not on the actual values of the $a_i.$ He further verified the conjecture in the special case of a binary alphabet. 
In this paper, we confirm Ramharter's conjecture for sets $\A$ with $|\A|=3$ and give an algorithmic procedure for constructing the unique maximizing arrangement. We also show that Ramharter's conjecture fails  for sets with $|\A|\geq 4$ in that the maximizing arrangement is in general neither unique nor independent of the values of the digits in $\A.$ The central idea is that  the extremal arrangements satisfy a strong combinatorial condition. This combinatorial condition may also be stated more or less verbatim in the context of infinite sequences on an ordered set.  We show that in the context of bi-infinite binary words,  this condition coincides with the Markoff property,  discovered by A.A. Markoff in 1879  in his study of minima of binary quadratic forms. We  further show that this same combinatorial condition is the fundamental  property which describes the orbit structure of the natural codings of points under a symmetric $k$-interval exchange transformation.\\ 

\textbf{Keywords:} Continued fractions, extremal values of continuants, Markoff property, Sturmian words and interval exchange transformations.\\

\textbf{MSC[2010]: } Primary 11J70; 37B10, 68R15.
\end{abstract}

\section{Introduction}

Given a finite sequence of positive integers $y=y_1,y_2,\ldots , y_n$, describe an arrangement or permutation $x=x_1,x_2,\ldots , x_n$  of the sequence $y$ which maximizes (resp.~minimizes) the regular continuant $K(x)=K_n(x_1,x_2,\ldots,x_n)$.  The continuant $K_n(x_1,x_2,\ldots,x_n)$ is defined recursively by $K_0()=1$, $K_1(x_1)=x_1$ and
 \begin{equation}\label{K1} K_n(x_1,x_2,\ldots,x_n)=x_nK_{n-1}(x_1,x_2,\ldots ,x_{n-1}) + K_{n-2}(x_1,x_2,\ldots ,x_{n-2})\end{equation} and is equal to the denominator of the finite regular  continued fraction $[0; x_1,x_2,\ldots ,x_n]$. 
This problem appears to be first attributed to C.A. Nicol (see \cite{MS}). There is no reason a priori that such an extremal arrangement should be unique. In fact, since 
$K_n(x_1,x_2,\ldots,x_n)=K_n(x_n,\ldots,x_2,x_1)$, the reversal (or mirror image) of any extremal arrangement is again extremal. But more generally, the function $K(\cdot)$ is far from being injective and it happens that many different permutations of the sequence $y$ have the same $K$ value \cite{RamZ}. 
There are many open questions concerning the distribution of the continuants $K_n(x_1,x_2,\ldots,x_n)$ ($n\in \nats)$ where the $x_i$ are restricted to a bounded subset of positive integers, including the famous Zaremba conjecture \cite{Zar}. The distribution of the continuants with the $x_i$ belonging to a bounded subset $\A$ is also extremely relevant in estimating the Hausdorff dimension of the set $E_{\A}\subset \reals $ consisting of all finite and infinite (regular) continued fractions whose partial quotients all belong to $\A$ (see for instance \cite{Cu1,Cu2,Good}).

T.S. Motzkin and E.G. Straus~\cite{MS} provided a first partial answer to Nicol's question
in the special case in which $y_1,y_2,\ldots , y_n$, are pairwise distinct. In \cite{Cu2}, T.W. Cusick found the maximizing arrangement for an arbitrary sequence $y_1,y_2,\ldots , y_n$ consisting of $1$s and $2$s. But the general problem was settled by G. Ramharter \cite{Ram83}.  He gave an explicit description of both extremal arrangements and showed that they are unique (up to reversal) and independent of the actual values of  the digits (see Theorem 1 in \cite{Ram83}).  For example, if $y=y_1,y_2, \ldots ,y_n$ is given in the form $a_1^{n_1}\cdots a_k^{n_k}$ with $1\leq a_1<a_2<\cdots <a_k$ and $n_1+n_2+\cdots +n_k=n$, then the maximizing arrangement for $K(\cdot)$ is unique up to reversal and  given by:
\[a_k L_{k-1}a_{k-2}L_{k-3}\cdots a_1^{n_1}\cdots  a_{k-3}L_{k-2}a_{k-1}L_k\]
where $L_i=a_i^{n_i-1}$. The fact that the extremal arrangements are unique is quite surprising as is remarkable that they do not involve the ring structure of the integers but rather only the relative order of the digits involved. Ramharter's theorem was later reproved by C.~Baxa in \cite{Baxa} and used to prove a criterion for the transcendence of continued fractions whose partial quotients are contained in a finite set (see also \cite{Dav}).

Motivated by a question in diophantine approximation \cite{Ram82}, Ramharter also considered the analogous problem in the context of the semi-regular continuant 
 $\Ks_n(x_1,x_2,\ldots,x_n)$
defined recursively by $\Ks_0()=1$, $\Ks_1(x_1)=x_1$ and 
\begin{equation}\label{K2} \Ks_n(x_1,x_2,\ldots,x_n)=x_n\Ks_{n-1}(x_1,x_2,\ldots ,x_{n-1}) - \Ks_{n-2}(x_1,x_2,\ldots ,x_{n-2}).\end{equation} 
For semi-regular continuants, the digit $1$ needs to be excluded and in this case, $\Ks(x)$ is the denominator of the terminating semi-regular 
continued fraction
\[ [x]^{ \bullet}=\cfrac{1}{x_1 -\cfrac{1}{x_2 -\cfrac{1}{
\ddots-\cfrac{1}{x_n}}}} \]
Letting $X$ be the tridiagonal matrix \[ X=\left( \begin{matrix}
x_1 & 1 & 0 &\cdots &0 \\
1 & x_2 & 1 &\ddots&\vdots \\
 0& 1 & \ddots &\ddots &0 \\
 \vdots &\ddots&\ddots &x_{n-1}& 1\\
 0&\cdots&0&1&x_n
  \end{matrix} \right)\] 
  we have that $K_n(x_1,x_2,\ldots , x_n)=\perm(X)$ while $\Ks_n(x_1,x_2,\ldots, x_n)=\det(X)$ where $\perm(X)$ (resp.~$\det(X)$) denotes the permanent (resp.~determinant) of the matrix $X$.
  
As in the case of the extremal arrangements for the regular continuant, Ramharter showed that the minimizing arrangement for $\Ks(\cdot)$ is unique (up to reversal) and independent of the choice of digits $a_1,\ldots ,a_k$. However, the determination of the maximizing arrangement for $\Ks(\cdot)$ turned out to be more difficult. In fact,  Ramharter points out that in contrast with the other three cases, there is an infinity of essentially different patterns and the maximizing arrangements for $\Ks(\cdot)$ must be described in terms of an algorithmic procedure as ``their combinatorial structure is exceptionally complicated.'' Ramharter conjectured that, as in the case of the other three extremal arrangements, for general sequences $a_1^{n_1}\cdots a_k^{n_k}$  the maximizing arrangement for the semi-regular continuant $\Ks(\cdot)$ is always unique (up to reversal) and independent of the values of the digits $2\leq a_1<a_2<\cdots <a_k$. The conjecture has been verified by Ramharter in two special cases : i) for sequence $y=y_1,y_2,\ldots ,y_n$  on a binary alphabet, i.e., $y_1,\ldots ,y_n \in \{a_1, a_2\}$ for some choice of positive integers $2\leq a_1<a_2$ (see theorems~1 and 2 in \cite{Ram83});  ii)  for sequences $y_1,y_2,\ldots,y_n$ with pairwise different entries (see Theorem~1 in \cite{Ram05}).


The central idea,  as was first observed by Ramharter in \cite{Ram83}, is that the extremal arrangements for both the regular and semi-regular continuants  satisfy a very special combinatorial condition. In the case of the maximising arrangement $x=x_1,x_2, \ldots,x_n$ (with each $x_i\geq 2)$ for the semi-regular continuant, this condition, denoted $\Ks_{\max},$ is as follows :
 \begin{description}
\item[$\Ks_{\max}$:] For each factorization $x=\overline u vw$ with $v\neq \overline v$ and $u\neq w$ one has that $v< \overline v$ if and only if $w<u$. 
\end{description}

Here $\overline u$ denotes the reversal of the sequence $u$, (i.e., if $u=u_1,u_2,\ldots, u_j$ then $\overline u= u_j, \ldots ,u_2,u_1),$ and  $<$ denotes the lexicographic order induced by the natural order on $\reals$ where however one takes the opposite convention of the true dictionary order with respect to proper prefixes, i.e., we declare $u<v$ whenever $v$ is a proper prefix of $u.$ We remark that condition $\Ks_{\max}$ makes sense for any finite sequence $x$ over any totally ordered set $\A,$ i.e., the entries of $x$ need not be positive whole numbers. 
As Ramharter points out, the central question is to understand whether for each sequence $y$ over an ordered alphabet $\A$ 
there exists a \emph{unique} permutation $x$ of $y$ verifying condition $\Ks_{\max}$,  which, if true,  would imply a unique global maximum for the semi-regular continuant. He further conjectured that this should be the case in general \cite{Ram05}. 

The following theorem confirms Ramharter's conjecture for sequences $y$ over any ternary ordered alphabet $\A :$ 

\begin{restatable}{mainthm}{Ksmax}
\label{gen}
Let $\A$ be any totally ordered ternary alphabet and $y=y_1,y_2,\ldots,y_n$ be any sequence with each $y_i\in \A$. Then there is precisely one permutation $x$ of $y$ (up to reversal) which verifies condition $\Ks_{\max}$. 
\end{restatable}

We also give an algorithmic procedure for constructing the permutation $x$ verifying condition $\Ks_{\max}.$ 
It follows that for each sequence $y=a_1^{n_1}a_2^{n_2}a_3^{n_3}$  of positive integers with $2\leq a_1<a_2<a_3, $ there exists a unique (up to reversal) permutation $x$ of $y$ which maximizes the semi-regular continuant $\Ks(\cdot)$. Moreover the maximizing arrangement is independent of the actual values of the digits $a_1,a_2$ and $a_3$ and depends only on the vector $(n_1,n_2,n_3)$.
Thus for sequences $y$ involving at most three distinct positive integers, condition $\Ks_{\max}$ gives a full characterization of the maximizing arrangement for the semi-regular continuant $\Ks(\cdot)$.

However,  for sequences $y$ over ordered alphabets $\A$ of cardinality $4$ or more, it may happen that $y$ admits more than one permutation verifying condition $\Ks_{\max}$. This in itself does not disprove the existence of a unique global maximum for the semi-regular continuant $\Ks(\cdot)$ which is independent of the actual choice of positive integers assigned to the elements of $\A$, however, it provides a basis for constructing such examples. In fact, we show the existence of a sequence $y$ over an ordered alphabet $\A=\{a<b<c<d\}$ having two permutations $x$ and $x'$ (with $x'\neq \overline x)$ each verifying $\Ks_{\max}$ and  depending on the values assigned to each of $a,b,c$, and $d$ in the sequence $y$, 
the maximum of $\Ks(\cdot)$ is assumed at $x$ and not at $x'$, or inversely the maximum occurs at $x'$ and not $x$ or the maximum occurs simultaneously at both $x$ and $x'$ (see Example~\ref{counter}).  It follows from this that for sequences $y$ involving $4$ or more distinct positive integers, $\Ks_{\max}$ is not a strong enough condition to guarantee a maximizing arrangement. 

A primary objective of this paper is to study condition $\Ks_{\max}$ in its full generality. 
Given a totally ordered set $\A$, let $\A^+$ denote the free semigroup generated by $\A$ consisting of all finite words $x=x_1x_2\cdots x_n$ with each $x_i\in \A$. We let $\A^*=\A^+ \cup \{\varepsilon\}$ be the free monoid generated by $\A$ in which we adjoin the empty word $\varepsilon$ regarded as the unique word of length $0$. We also let $\A^\nats$ (resp.~$\A^\ints)$ denote the set of all one sided (resp.~two-sided) infinite words $x=x_1x_2x_3\cdots$ (resp.~$x=\cdots x_{-2}x_{-1}x_0x_1x_2\cdots)$ with each $x_i \in \A$. 

We introduce the following definition which amounts to a reformulation of  $\Ks_{\max}$ to a more general setting which includes both finite and infinite words:

\begin{maindef}\label{sing1}
Let $\A$ be any totally ordered alphabet, and $x\in \A^+\cup \A^\nats \cup \A^\ints. $ We say that $x$ is \emph{singular} if for all
factorizations $x=\overline u vw$
$(v\in \A^+, u,w \in \A^*\cup \A^\nats)$
with $v\neq \overline v$ and $u\neq w$ we have $v<\overline v$ if and only if $w<u$.
\end{maindef}

The order $<$ is taken to be the  lexicographic order on $\A^*\cup \A^\nats$  induced by the order on $\A$ (where as in $\Ks_{\max},$ we take the opposite convention of the true dictionary order when it comes to proper prefixes).   

In \cite{Ram05, Ram06}, Ramharter showed that each finite singular word $x$ over a binary alphabet $\A=\{a,b\}$ is a finite Sturmian word, and that the palindromic binary maximizing arrangements for $\Ks(\cdot)$ are in one-to-one correspondence with the extremal cases of the Fine and Wilf theorem \cite{FW} with two co-prime periods (see Theorem~3 in \cite{Ram05}). 
 We give a more precise reformulation of this result in terms of more standard notions in the theory of finite Sturmian words and already existing algorithms: 

\begin{restatable}{mainthm}{finitebinary}
\label{t:Chr}
A finite word $x$ over a binary ordered alphabet $\{a<b\}$ is singular if and only if $x$ or $\overline x$ is of the form $b^n,$ $ab^n$ ($n\geq 0)$ or $aya$ where $ayb$ is a power of a Christoffel word $C_{p,q}$ with co-prime periods $p$ and $q.$
\end{restatable}

If we now consider bi-infinite words $x$ over a binary alphabet $\A=\{a,b\}$, then we show that the singular property in \Cref{sing1} coincides with the so-called Markoff property,  first identified by A.A. Markoff in \cite{Ma1} in his study of minima of binary quadratic forms, and again in \cite{Ma3} (see in particular page 28) in which he answers a question posed by J.~Bernoulli in \cite{B}. The Markoff property applies to bi-infinite words $x$ over a binary alphabet $\A=\{a,b\}$ and may be formulated as follows:

\begin{enumerate}
\item[(M)] For each factorization $x=\overline u a'b'w$ with $\{a',b'\}=\{a,b\}$ and $u,w \in \A^\nats$, either $u=w$ or if $i\in \nats$ is the least index $j$ for which $u_j\neq w_j$, then $u_i=b'$ and $w_i=a'$.  \end{enumerate}

Although both $\Ks_{\max}$ and property (M) were discovered in the context of extremal problems involving continued fractions, the former relates to semi-regular continuants while the latter concerns inequalities relating real numbers and convergents of infinite continued fractions. So the connection between the two is not fully transparent. 
As was first observed by T.W. Cusick and M.E. Flahive in \cite{CF} and later proved by C. Reutenauer in \cite{Reu}, property (M) is equivalent to the balanced property of Morse-Hedlund \cite{MH}. 



We show  that a bi-infinite binary word verifies the Markoff property (M) if and only if it is singular (see Corollary~\ref{MP}).
One way to establish the equivalence is to show that a bi-infinite binary word is singular if and only if it is balanced (see Proposition~\ref{fullbal}) and then use Theorem~3.1 in \cite{Reu}. Alternatively, a direct proof (which bypasses the balance property) may be given using Sturmian morphisms. As for one sided-infinite binary words, we show 

\begin{restatable}{mainthm}{LyndSt}
\label{Sturmian}
Let $x\in \{a,b\}^\nats$ be an aperiodic binary (one sided) infinite word. Then $x$ is singular if and only if $x$ is a Sturmian Lyndon word.
\end{restatable}

Recall that a finite or infinite word $x$ over an ordered alphabet is said to be \emph{Lyndon} if and only if $x$ is lexicographically smaller than each of its proper suffixes.
 
The next question is to understand bi-infinite singular words over higher alphabets. We show that the singular property  is the fundamental property which distinguishes the orbit structure of codings of symmetric interval exchange transformations from other subshifts of the same factor complexity including Arnoux-Rauzy subshifts \cite{AR}. A {\em symmetric $k$-interval exchange transformation} $\mathcal I$ is given by a probability vector of $k$ lengths $(\alpha_1,\ldots ,\alpha _k)$. The unit interval is partitioned
into $k$ subintervals of lengths $\alpha _1,\ldots ,\alpha_k$ labeled $1,2,\ldots ,k$ which are then re-arranged according to the permutation $\sigma (j)=k+1-j$. A natural coding of a point $x\in [0,1]$ under $\mathcal I$  is given by a bi-infinite word $(x_n)_{n\in \ints}$ over the alphabet $\{1,2,\ldots ,k\}$ where $x_n=i$ whenever  the $n$-th iterate of $x$ lies in $i$th interval. By the language of $\mathcal I$ we mean the language defined by all natural codings under $\mathcal I.$ We obtain the following characterization of singular bi-infinite words $x$ whose associated language is \emph{symmetric} (i.e., closed under reversal) and which constitutes a generalization of Theorem~3.1 in \cite{Reu}:

\begin{restatable}{mainthm}{symmiets}
\label{ietssymm}
Let $\A_k=\{1,2,\ldots ,k\}$ $(k\geq 2)$ and let $x\in \A_k^\ints$ be uniformly recurrent and assume that each $i\in \A_k$ occurs in $x.$ Then the following are equivalent:
\begin{enumerate}
\item $L(x)$ is the language of a symmetric $k$-interval exchange transformation.
\item $x$ is singular and $L(x)$ is symmetric. 
\end{enumerate}
\end{restatable}

As in \Cref{Sturmian}, if we take $x$ to be a one-sided infinite word, then in item 1 of \Cref{ietssymm} we need to add the condition that $x$ is Lyndon. We also give a characterization of natural codings of symmetric $k$-interval exchange transformations satisfying Keane's infinite distinct orbit condition (or i.d.o.c.) in terms of infinite singular words.

\section{Preliminaries}

Let $\A=\{a_1,a_2,\ldots ,a_k\}$ be a  totally ordered set with $a_1<a_2<\cdots <a_k$. For each $x\in \A^+$ and $a\in \A$ we let $|x|_a$ denote the number of occurrences of the letter $a$ in $x$ and write $|x|=\sum_{a\in \A} |x|_a$ for the length of $x$. The \emph{Parikh vector} of $x$ is defined by $(|x|_{a_1},|x|_{a_2},\ldots,|x|_{a_k})$.  Two finite words $x,y \in \A^+$ are said to be \emph{abelian equivalent} if they define the same Parikh vector, i.e., $|x|_a=|y|_a$ for each $a\in \A$. We define the abelian class of a word $x\in \A^+$ to be the set of all words $y\in \A^+$ which are abelian equivalent to $x$. 
If $x\in \A^\nats \cup \A^\ints, $ then a factor $u$ of $x$ is said to be recurrent in $x$ if every suffix of $x$ contains an occurrence of $u$. We say $x$ is recurrent if every factor $u$ of $x$ is recurrent in $x$. 
For $x\in \A^+\cup \A^\nats \cup \A^\ints$, we let $L(x)$ denote its \emph{language}, i.e., the set of all factors of $x$.
We say $x$ is \emph{balanced} if for all factors $u,v$ of $x$ with $|u|=|v|$ we have $||u|_a-|v|_a|\leq 1$ for each $a\in \A$. For all other word combinatorial definitions which have not been explicitly defined herein, we refer the reader to \cite{Lot}. 

The total order $\leq $ on $\A$ defines a lexicographic order on $\A^*\cup \A^\nats$, also denoted $\leq $, where we write $x<y$ if either $y$ is a proper prefix of $x$ (so we are adopting the opposite convention of the true dictionary order) or if $x=ua_ix'$, $y=ua_jy'$ for some $u\in \A^*$, $x',y' \in \A^*\cup \A^\nats$ and $i<j$. For $x,y \in \A^*\cup \A^\nats$, we write $x\leq y$ if either $x=y$ or $x<y$.

\begin{defin}Let $x \in \A^*\cup \A^\nats \cup \A^\ints$. A factorization $x=\overline uvw$  (with $v\in \A^+$, $u,w\in \A^*\cup \A^\nats)$ and with $v\neq \overline v$ and $u\neq w$ is called \emph{reversible} (resp.~\emph{singular}) if  $v< \overline v$ and $u< w$ or $\overline v< v$ and $w<u$ (resp.~if $v< \overline v$ and $w< u$ or $\overline v< v$ and $u<w)$. We say $x$ is \emph{reversible} if $x$ admits a reversible factorization.
\end{defin}

It follows immediately that each word $x\in  \A^*\cup \A^\nats \cup \A^\ints$ is either reversible or singular, but not both. 
A reversible word may admit more than one reversible factorization. For example, if $\A=\{a,b\}$ with $a<b$, then the reversible word $x=aabb$ admits the reversible factorization $x=(a)(ab)(b)$ as well as $x=(a)(abb)\varepsilon$.
We also note that the property of being reversible (resp.~singular) is invariant under reversal, i.e.,  $x$ is reversible (resp.~singular) if and only if $\overline x$ is reversible (resp.~singular). To check that a word $x$ is singular, it suffices to show that any factorization $x=\overline u vw$ with $u\neq w$ and where $v$ begins and ends in distinct letters is singular.  In fact, writing $v=\overline zcyd z$ with $z,y \in \A^*$ and $c,d\in \A$ distinct, we see that if the factorization $x=\overline u v w$ is reversible, then so is the factorization  $x=\overline{ z u} (cyd) zw$.%

 \section{Finite singular words}

 We begin this section with a characterization of finite binary singular words. We then study finite ternary singular words and show that corresponding to each Parikh vector $\pv=(n_a,n_b,n_c)$ there exists a unique (up to reversal) singular word $x$ over the ordered alphabet $\{a<b<c\}$ whose Parikh vector is equal to $\pv$. We also describe an algorithm for constructing the singular word $x$. Finally, we show that on ordered alphabets of size greater than three, there may exist many abelian equivalent singular words. In particular, we show that the abelian class defined by the ordered Parikh vector $(n_a,n_b,n_c,n_d)=(1,2,1,2)$ over the ordered alphabet $\{a<b<c<d\}$ contains up to reversal two distinct singular words $x$ and $x'$ and that the maximizing arrangement for $\Ks(\cdot)$ depends on the actual values of the positive integers assigned to each of the letters $a,b,c$ and $d$.  
 
We begin with a simple lemma which applies to arbitrary ordered alphabets $\A$.
For a word $x\in\A^+$, let $\min x$ denote the smallest letter occurring in $x$.

\begin{lemma}
    \label{t:a1st}
    Let $\A$ be a totally ordered alphabet, $x\in \A^+$ be a singular word, and
    $a=\min x$. Then $x$ must either begin or end in $a$. If $x=ax'$
    (resp.~$x=x'a$) for some $x'\in\A^+$, then $x'$ ends (resp.~begins) with
    $\min x'$.
\end{lemma}
\begin{proof}
    For $b\in \A$ and $u,v\in \A^*$, the factorization $x=\overline u\cdot avb\cdot\eps$ is 
    reversible if $u\neq \eps$ and $b\neq a.$  Thus $x$ has $a$
    as its first or last letter. Short of replacing $x$ by $\overline x$, we
    may assume without loss of generality that 
    $x=avb=ax'$. If $v=\overline{u'}cv'$ with $c\in\A$, then since
    the factorization $a\overline{u'}\cdot cv'b\cdot\eps$ is not reversible, we obtain
    $c\geq b$. Hence $b=\min x'$.
\end{proof}

\begin{rem}
\label{t:ends}
Note that \Cref{t:a1st} implies that for a singular word $x,$ if 
$x=sat$ with $s,t\in \A^+$ and $a=\min \A,$ then 
$a$ is a prefix of $s$ and a suffix of $t;$  whereas if $x$ begins in $c$ where $c=\max\A,$ then $x=c^nb$ for some $n\geq 1$ and $b\in \A.$

\end{rem}

\noindent We now study binary singular words.
  
 \begin{lemma}\label{bal} Let $\A$ be an ordered binary alphabet $\{a<b\}$ and let $x\in \A^*\cup \A^\nats\cup \A^\ints$ be singular. Then $x$ is balanced. 
 \end{lemma}
 
 \begin{proof} Let $x$ be singular and assume to the contrary that $x$ is not balanced. By Proposition~2.1.3 in \cite{Lot},
 there exists a palindrome $z\in \A^*$ such that both $aza$ and $bzb$ are factors of $x$. Since the factors $aza$ and $bzb$ cannot overlap one another, there exist distinct $a',b'\in \A$ and a factorization  $x=\overline u v w$ where $v$ begins in $a'$ and ends in $b'$, $za'$ is a prefix of $u$ and $zb'$ a prefix of $w$. It follows that $a'<b'$ if and only if $u<w$ and hence we obtain a reversible factorization of $x$, a contradiction. \end{proof}
 
 \begin{thm}\label{Stbispecial} A finite word $x$ over a binary ordered alphabet $\A=\{a<b\}$ is singular if and only if $x$ or $\overline x$ is of the form $b^n,$ $ab^n,$ $(n\geq 0)$ or $aya$ where $y$ is a bispecial Sturmian word.
\end{thm}

A word $y\in\{a,b\}^*$ is a {\it bispecial Sturmian word} if $ay,by,ya$ and $yb$ are each a factor of an infinite Sturmian word, or equivalently if they are each balanced (see Proposition~2.1.17 in \cite{Lot}).  

\begin{rem} If $x\in\{a,b\}^\nats$ is a Sturmian word, then the bispecial factors of $x$ are called \emph{central words} and are  bispecial Sturmian words. The factors of $x$ of the form $aya$, with $y$ a central word, are already known in the literature and called \emph{singular Sturmian words} \cite{WW}. The reason for the term \emph{singular} is that if $n=|y|+2$, then the $n+1$ factors of $x$ of length $n$ are partitioned into two distinct abelian classes, one class contains $n$ words (which are all cyclically conjugate to the Christoffel word $ayb$) while the other class consists only of the word $aya$. Thus relative to the order $a<b,$ every singular Sturmian word of the form $aya$ is singular in the sense of \Cref{sing1}, but not every singular binary word of the form $aya$ is a singular Sturmian word  since a bispecial Sturmian word need not be central (e.g. $ab$).\end{rem}
 
 \begin{proof}[Proof of \Cref {Stbispecial}]
 Assume first that $x\in\A^+$ is singular. By \Cref{t:a1st}, if $|x|_a=1$, then
 $x=ab^n$ or $x=b^na$ with $n=|x|_b$, whereas if $|x|_a\geq 2$, then we may
 write $x=aya$ for some $y\in\A^*$. To see that $y$ is a bispecial Sturmian
 word, it suffices to show that each of $ay,by,ya$ and $yb$ is balanced (see
 Proposition~2.1.17 in \cite{Lot}). Since $x$ itself is balanced by \Cref{bal},
 it suffices to show that $by$ and $yb$ are balanced. Let us assume to the
 contrary that $by$ is not balanced, and pick a palindrome $z$ such that $aza$
 and $bzb$ are factors of $by$. Since $y$ is balanced, it follows that $bzb$ is
 a prefix of $by$. 
 Thus we can write $by=bz\cdot v\cdot zay'$ with $y'\in\A^*$ and where $v$
 begins in $b$ and ends in $a$.  This gives a factorization $x=\overline u v w$
 where $u=za$ and $w=zay'a$. It follows that  $\overline v <v$ and $w<u$ (since $u$ is a
 proper prefix of $w$) whence $x$ is reversible, a contradiction.
 Similarly, one shows that $yb$ is also balanced. 
 
 For the converse, we first note that any word of the form $b^n$ or $ab^n$ is clearly singular. So let us assume that $x=aya$ with $y$ a
 bispecial Sturmian word.
 To see that $x$ is singular, fix a factorization $x=\overline uvw$ with $v=a'v'b',$  $u\neq w$ and $\{a',b'\}=\A.$ Short of replacing $x$ by $\overline x$ we may assume without loss of generality that $v<\overline v$, whence $v=av'b.$ We will show that $w<u.$
 If $u$ is a proper prefix of $w$, then clearly $w<u$.  So let us next suppose that $w$ is a proper prefix of $u$.  If $w$ is empty, then $v=av'b$ is a suffix
 of $x$, which contradicts $x=aya$.
 If $w\neq\eps$ then we may write $w=ra$ and $u=rau'$ for some $r\in \A^*$ and
 $u'\in \A^+$. It follows that $a\overline r vw=a\overline rav'bra$ is a proper
 suffix of $x$ and hence $a\overline rav'brb$ is a suffix of $yb$, 
 contradicting that $yb$ is balanced.
 
 Finally, if neither $u$ nor $w$ are proper prefixes of one another, then we may write $u=rcu'$ and $w=rdw'$ where $r,u',w' \in \A^*$ and
 $\{c,d\}=\A$. Thus
 \[x=\overline u v w=\overline{u'}c\overline r av'brdw'.\]
 If $w'$ is empty, then $d=a$ and so $c=b$ and $w<u$. If $w'\neq\eps$, then
 $c\overline r av'brd$ is a factor of $ay$ and since $ay$ is
 balanced, it follows that $d=a$ and $c=b,$ whence  $w<u$ as required.
 \end{proof} 
 
\noindent As a consequence, we can now prove
\finitebinary*
 
\begin{proof} This follows immediately from the fact that a word $y\in\{a,b\}^*$ is a bispecial Sturmian word if and only if $ayb$ or
$a\overline yb$ is a power of a Christoffel word $C_{p,q}$ with $\gcd(p,q)=1$ (see Theorem~3.11 in \cite{Fici}). \end{proof}

\begin{example} \Cref{t:Chr} may be used to construct binary singular words having a specified Parikh vector $(n_a,n_b)$ via standard algorithms for constructing Christoffel words. For example, suppose we want to find a singular word $x\in \{a,b\}^+$ whose corresponding Parikh vector is $(7,14)$.
Then we shall first build a word of the form $x'=ayb$ (with $y$ a Sturmian bispecial) with Parikh vector $(6,15)$ and then $x=aya=x'b^{-1}a$.
To construct $x'$, we first note that $\gcd(6,15)=3$ and hence we must find the Christoffel word $C_{2,5}$. Using standard algorithms (see~\cite{BLRS})
we find that $C_{2,5}=abbabbb$. It follows that $x'=C_{2,5}^3=abbabbb\cdot abbabbb\cdot abbabbb$ and hence $x=  abbabbb\cdot abbabbb\cdot abbabba$.
\end{example}

Before moving to the ternary case, let us give a few more results which apply
to arbitrary ordered alphabets.
We say that a letter $a$ is \emph{separating} for a word $x\in \A^+$ if $a$
occurs in every factor of $x$ of length $2$. Notice that two distinct letters
$a,b$ are both separating for $x$ if and only if $x$ is one of the following,
for a suitable $n\geq 0$: $(ab)^n$, $(ba)^n$, $(ab)^na$, or $(ba)^nb$.
\begin{lemma}
    \label{t:sep}
    Let $d\in\{\min \A,\max \A\}$. If a singular word $x$ has $dd$ as a
    factor, then $d$ is separating for $x$.
\end{lemma}
\begin{proof}
    Assume to the contrary that $x$ contains both $dd$
    and $ab$ as factors, with $a,b\in \A\setminus\{d\}.$  We may assume without loss
    of generality that $dd$ occurs in $x$ before $ab$, and hence  
    $x=\bar ud\cdot dva\cdot bw$ for suitable $u,v,w\in \A^*$. As $d$ is either
    smaller or bigger than both $a$ and $b$, this factorization is
    reversible, a contradiction.
\end{proof}


\begin{lemma}
	\label{t:asep}
    Let $a=\min\A$, and $x\in\A^*$ be a singular word with $|x|_a>1$.
    Then $a$ is separating in $x$ if and only if
    \begin{equation}\label{ineqsep}
        |x|_a\geq\sum_{a'\in\A\setminus\{a\}}|x|_{a'}+1,
    \end{equation}
    and $aa\in\F(x)$ if and only if the inequality is strict.
\end{lemma}
\begin{proof}
	By \Cref{t:a1st}, $x$ begins and ends with $a$. Thus if $a$ is
	separating for $x,$ then every occurrence in $x$ of a letter $a'\neq a$ must be immediately preceded by $a.$  In particular, if $a$ is separating,
	then (\ref{ineqsep}) holds.
	
 	Conversely, assume (\ref{ineqsep}) holds, and suppose to the contrary that $a$
 	is not separating for $x$. This means that $x$ has at least one factor
 	$bc$ with $b,c\in\A\setminus\{a\}$. But then (\ref{ineqsep}) implies
 	$aa\in\F(x)$, which is absurd in view of \Cref{t:sep}.
    It is now clear that $aa$ occurs if and only if the inequality in (\ref{ineqsep}) is strict.
\end{proof}

\noindent The following analogue for $\max\A$ can be proved in a similar way.
\begin{lemma}
	\label{t:csep}
	Let $c=\max\A$, and $x\in\A^*$ a singular word. Then $c$ is separating for
	$x$ if and only if
	\[|x|_c\geq\sum_{c'\in\A\setminus\{c\}}|x|_{c'}-1.\]
	When $|x|-|x|_c>1$, the above inequality is strict if and only if
	$cc\in\F(x)$.
\end{lemma}
%
%

For a letter $d\in \A$, let $\lambda_d$ and $\rho_d$ denote the morphisms
defined by
$\lambda_d(d')=dd'$, $\rho_d(d')=d'd$ for letters $d'\neq d$, and
$\lambda_d(d)=\rho_d(d)=d$.
The following properties are easy to verify and will be useful in the sequel:
for all $d\in \A$ and $x\in \A^*$,
\begin{equation}
    \label{e:lamrho}
    d\overline{\lambda_d(x)}=\lambda_d(\overline x)d\quad\text{ and }\quad
    d\rho_d(x)=\lambda_d(x)d.
\end{equation}

\begin{lemma}
	\label{t:monoa}
	Let $a\in\A$ and $x,y\in \A^*$. Then
	$x<y\iff\lambda_a(x)a<\lambda_a(y)a$.
\end{lemma}
\begin{proof}
	If $y$ is a proper prefix of $x$, then clearly $\lambda_a(y)a$ is a
	prefix of $\lambda_a(x)$, and hence a proper prefix of $\lambda_a(x)a$.
	Conversely, if $\lambda_a(x)a=\lambda_a(y)aWa$ for some $W\in \A^*$, then there exists $w\in \A^+$ such that  $aW=\lambda_a(w)$ and $x=yw$.
	
	If $x=ubv$ and $y=ucw$ for some $u,v,w\in \A^*$ and $b,c\in \A$ such that
	$b<c$, then in all cases
	$\lambda_a(x)a$ begins with $\lambda_a(u)ab$ and
	$\lambda_a(y)a$ begins with $\lambda_a(u)ac$. Conversely, assume
	$\lambda_a(x)a<\lambda_a(y)a$, the latter not being a prefix of the
	former. Their longest common prefix is then necessarily
	$\lambda_a(u)a$ for a suitable $u\in \A^*$. Hence there exist letters
	$b<c$ such that $\lambda_a(u)ab$ and $\lambda_a(u)ac$ are respectively
	prefixes of $\lambda_a(x)a$ and $\lambda_a(y)a$; therefore $x<y$, as $x$
	begins with $ub$ and $y$ with $uc$.
\end{proof}
For $c=\max \A$ we need the following slight variation:

\begin{lemma}
	\label{t:monoc}
	Let $c=\max \A$ and $x,y\in \A^*$. Then
	$x<y\iff\lambda_c(x)<\lambda_c(y)$.
\end{lemma}
\begin{proof}
	If $y$ is a proper prefix of $x$, then $\lambda_c(y)$ is a proper prefix
	of $\lambda_c(x)$. Conversely, if $\lambda_c(y)W=\lambda_c(x)$ for some
	nonempty $W$, then either $y$ is a proper prefix of $x$, or
	$y=uc$ and $x=uax'$ for some $u,x'\in \A^*$ and a letter $a<c$.
	Hence $x<y$ anyway.
	
	If $x=uav$ and $y=ubw$ for suitable $u,v,w\in \A^*$ and $a,b\in \A$ with
	$a<b$, then either $c\notin\{a,b\}$ or $b=c$. In both cases
	$\lambda_c(x)$ begins with $\lambda_c(u)ca$. As for $\lambda_c(y)$, it
	equals $\lambda_c(u)c$ if $b=c$ and $w=\eps$, otherwise it has
	$\lambda_c(u)cb$ as a prefix.
	Conversely, if $\lambda_c(x)<\lambda_c(y)$ and their longest
	common prefix is shorter than $\lambda_c(y)$, then it necessarily ends in
	$c$. In other words, there exist $u\in \A^*$ and letters $a<b$ such that
	$\lambda_c(u)ca$ and $\lambda_c(u)cb$ are prefixes of $\lambda_c(x)$ and
	$\lambda_c(y)$ respectively. Again, it follows that $x$ begins with
	$ua$, and $y$ with $ub$.
\end{proof}

\begin{lemma}
    \label{t:dAl}
    Let $a=\min \A$ and $x\in \A^*$. Then $x$ is singular if and only if  $\lambda_a(x)a$ is singular.
\end{lemma}
\begin{proof}
    Equivalently, it suffices to show that $x$ is reversible 
    if and only if $\lambda_a(x)a$ is reversible. Fix a factorization $x=\overline{u}vw$ with $v\neq \overline v$ and $u\neq w$; this
    is equivalent to
    \begin{equation}
	    \label{e:lama}
    	\lambda_a(x)a
	    =\overline{\lambda_a(u)a}\cdot a^{-1}\lambda_a(v)\cdot\lambda_a(w)a
    \end{equation}
	since \cref{e:lamrho} implies
    $\lambda_a(\overline{u})a=a\overline{\lambda_a(u)}
    =\overline{\lambda_a(u)a}$. Again by \cref{e:lamrho} we have
	$\overline{a^{-1}\lambda_a(v)}=a^{-1}\lambda_a(\overline{v})$. Since
	$|a^{-1}\lambda_a(v)|=|a^{-1}\lambda_a(\overline{v})|$, we obtain
	\[a^{-1}\lambda_a(v)<\overline{a^{-1}\lambda_a(v)}
	\iff\lambda_a(v)a<\lambda_a(\overline{v})a\]
	as well as the same equivalence for the opposite inequality. 
	By applying
	\Cref{t:monoa} to $u,w$ and to $v,\overline{v}$, it follows that
	$x=\overline{u}vw$
	is a reversible factorization if and only the one in \cref{e:lama} is
	reversible too.
	
	It remains to show that if $\lambda_a(x)a$ has a reversible
	factorization, then it has one of the form shown in \cref{e:lama}. Let
	then $\lambda_a(x)a=\overline{U}VW$ be a reversible factorization. We may
	assume that the first letter of $V$ differs from its last; in particular,
	by symmetry we may assume that the first letter of $V$ is not $a$, in which case $U$
	must begin with $a.$  Now, $W$ cannot be empty, otherwise $V$
	would end in $a$ and the factorization would not be reversible, as
	$\overline{V}<V$ and $U<W$; for the same reason, $W$ must begin with $a$.
	It follows that there exist $u,v,w\in\A^*$ such that the factorization
	$\lambda_a(x)a=\overline{U}VW$ coincides with the one in \cref{e:lama}, as
	desired.	
\end{proof}

\begin{lemma}
	\label{t:dCl}
	Let $c=\max\A$ and $x\in\A^+$. Then $x$ is singular if and only if
	$\rho_c(x)c^{-1}$ is singular.
\end{lemma}
\begin{proof}
	Once again, we prove the contrapositive, i.e., $x$ is reversible
	if and only if $\rho_c(x)c^{-1}$ is reversible. Let
	$x=\overline{u}vw$ for some $u,v,w\in\A^*$ with $v\neq\eps$, which is
	equivalent to
	\begin{equation}
		\label{e:lamc}
		\rho_c(x)c^{-1}=c^{-1}\lambda_c(x)
		=\overline{\lambda_c(u)}\cdot c^{-1}\lambda_c(v)\cdot\lambda_c(w)
	\end{equation}
	in view of \cref{e:lamrho}, which also implies
	$\overline{c^{-1}\lambda_c(v)}=c^{-1}\lambda_c(\overline{v})$. Since
	clearly $c^{-1}\lambda_c(v)$ is less (resp.~greater) than
	$c^{-1}\lambda_c(\overline{v})$ if and only if the same relation holds
	between $\lambda_c(v)$ and $\lambda_c(\overline{v})$, applying
	\Cref{t:monoc} to the pairs $u,w$ and $v,\overline{v}$ we obtain that
	$x=\overline{u}vw$ is reversible if and only if so is the factorization in
	\cref{e:lamc}.
	
	It remains to show that if $\rho_c(x)c^{-1}$ has a reversible
	factorization, then it has one like in \cref{e:lamc}. Indeed, if
	$\rho_c(x)c^{-1}=\overline{U}VW$ is reversible, we may assume that $V$
	begins and ends with distinct letters, and by symmetry that $V$ does not
	begin with $c.$ It follows that $U$ is either empty or begins with $c$. Moreover,
	if $W$ is nonempty then it also must begin with $c$, since otherwise $V$
	would end in $c$, yielding $W<U$ and $V<\overline{V},$ a contradiction. It follows that
	there exist $u,v,w\in\A^*$ such that $\overline{U}\cdot V\cdot W$
	coincides with the factorization in \cref{e:lamc}.
\end{proof}

Since Christoffel words are exactly the images of $a,b$ under compositions of
$\lambda_a$ and $\rho_b$ (see for instance~\cite{BLRS,dadl2}), an alternative
proof of \Cref{t:Chr} can easily be obtained from \Cref{t:dAl,t:dCl}.

We now focus on a ternary alphabet. In the following, we let $\A=\{a,b,c\}$
with $a<b<c$.

\begin{rem}
    \label{t:nospecial}
    Let $x\in\A^*$. If $ab,cb\in\F(x)$, or if $ba,bc\in\F(x)$, then $x$ is not
    singular. Indeed, $\overline{u}a\cdot bvc\cdot bw$, $\overline{u}c\cdot bva\cdot bw$
    and their reverses are all reversible factorizations.%
\end{rem}
\begin{lemma}
    \label{t:ab+bc}
    Let $x\in\A^*$ be a singular word. If $ab$ and $bc$ (resp.~$cb$ and $ba$)
    occur in $x$, then $x=ab^n(ca)^m$ (resp.~$\overline{x}=ab^n(ca)^m$) for
    some $n,m\geq 1$.
\end{lemma}
\begin{proof}
Suppose $ab, bc\in L(x).$ As $ab\in L(x)$. it follows that $\{cb,cc\}\cap L(x)=\emptyset$. Similarly, $bc\in L(x)$ implies $\{aa, ba\}\cap L(x)=\emptyset$.
Write $x=\overline u\cdot c \cdot z$  where $z$ follows the last occurrence of $c$ in $x$. Then $|u|_b>0$ and by \Cref{t:a1st}, $z$ is nonempty and hence begins with $a$. We claim that $|z|_b=0$; indeed, $x$ cannot end in $b$ by \Cref{t:ends}, and the last $b$ in $x$ can only be followed by $c$.
It follows that $z=a$ and hence $ca$ is a suffix of $x$. Thus $x=wab^n(ca)^m$ for some $n,m\geq 1$, which implies $w=\eps$ otherwise the factorization
$w\cdot ab^n(ca)^m\cdot\eps$ would be reversible.
\end{proof}

\begin{prop}
    \label{t:delta0}
    Let $x\in\A^*$ be a singular word. If $|x|_a=|x|_c+1$, then $x=ab^n(ca)^m$
    or $\overline{x}=ab^n(ca)^m$ for some $n,m\geq 0$.
\end{prop}
\begin{proof}
    By \Cref{t:asep,t:csep}, $aa,cc\notin\F(x)$. Hence $|x|_b=0$ if and only
    if $x=a(ca)^m$ for some $m\geq 0$.
    Let then $|x|_b\geq 1$. If $|x|_a=1$, then $x=ab^n$ with $n=|x|_b$ and the
    assertion is again verified. Otherwise, $x$ begins and ends with $a$ by
    \Cref{t:a1st}.
    Since $|x|_a=|x|_c+1$ and $aa,cc\notin\F(x)$, it follows that $ab$ and
    $ba$ are both factors of $x$ if and only if so are $cb$ and $bc$. Since
    this is impossible in view of \Cref{t:nospecial}, the only remaining
    options are $ab,bc\in\F(x)$ or $cb,ba\in\F(x)$. The assertion then follows
    from \Cref{t:ab+bc}.
\end{proof}
\begin{lemma}
	\label{t:bruns}
    Let $x\in\A^*$ be a singular word such that $|x|-|x|_c>1$, and let
    $\delta=|x|_c-|x|_a+1$. The following hold:
    \begin{enumerate}
        \item If $\delta>0$ (resp.~$\delta<0$), then $ab$ and $ba$ (resp.~$bc$
        and $cb$) do not occur in $x$.
        \item If $|x|_b\geq|\delta|>0$, then $x$ has exactly $|\delta|$ runs
        of consecutive occurrences of $b$.
    \end{enumerate}
\end{lemma}
\begin{proof}
    Assume $\delta>0$ first, that is, $|x|_a<|x|_c+1$. Then $|x|_b<|\delta|$
    is equivalent to $|x|_c>|x|_a+|x|_b-1$, and then to $cc\in\F(x)$ by
    \Cref{t:csep}. In such a case, $c$ is separating for $x$, so that
    $ab,ba\notin\F(x)$. So assume that  $|x|_b\geq|\delta|,$ in which case
    $cc\notin\F(x)$. By \Cref{t:a1st}, $c$ is neither a prefix nor a suffix of
    $x$. By contradiction, suppose $ab\in\F(x)$, so that $cb\notin\F(x)$ in
    view of \Cref{t:nospecial}. Thus all occurrences of $c$ would be followed
    by $a$.
    As $ab\in\F(x)$, $a$ occurs in $x$ not only as a suffix; by \Cref{t:a1st},
    it must occur as a prefix too, so that $|x|_a\geq|x|_c+1$, contradicting
    $\delta>0$. Thus $ab\notin\F(x)$, and $ba\notin\F(x)$ by a symmetric
    argument. Hence, $a$ can be followed (and preceded) only by $c$
    within $x$. A simple counting argument then shows that $x$ has exactly
    $|x|_c-|x|_a+1=\delta$ runs of $b$.

    Now let $\delta<0$, so that $|x|_a>|x|_c+1$, and $x$ begins and ends with
    $a$ by \Cref{t:a1st}. By a similar argument as before, we obtain
    $cb,bc\notin\F(x)$, so that $c$ can be preceded and followed only by $a$.
    If $|x|_b\geq|\delta|$, that is, if $|x|_a\leq|x|_b+|x|_c+1$, we have
    $aa\notin\F(x)$ by \Cref{t:asep}. Therefore, we can count
    $|x|_a-|x|_c-1=|\delta|$ runs of $b$ in $x$.
\end{proof}

Let $\xi:\A^*\to\A^*$
be the map defined as follows: for all $x\in\A^*$, $\xi(x)$ is obtained by
adding an occurrence of $b$ to all existing runs of consecutive $b$ in $x$, as
well as in the middle of any occurrence of $aa$ or $cc$. Unlike the maps we
used for adding occurrences of $a$ and $c$, this $\xi$ cannot be realized via episturmian
morphisms, 
but it is \emph{sequential}, i.e., obtained as the output of a
sequential transducer (cf.~\cite{Lot}), as shown in \Cref{f:bseq}.
We remark that $\xi$ commutes with reversal, that is, for all $x\in\A^*$,
\[\overline{\xi(x)}=\xi(\overline{x}).\]
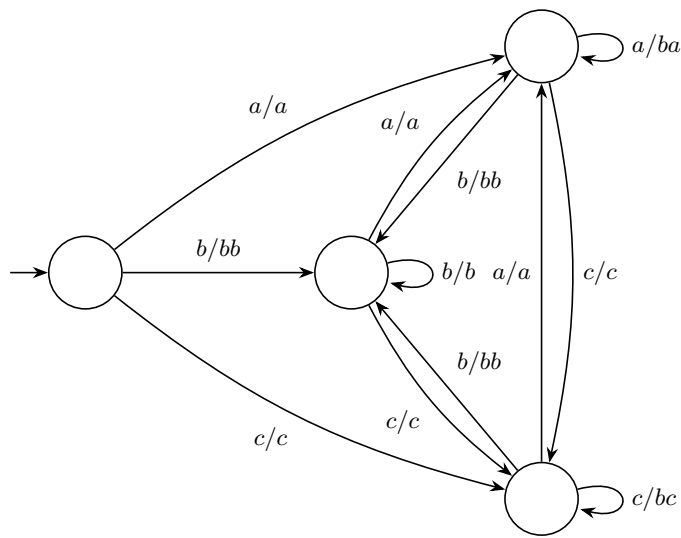
\begin{figure}[htb]
    \centering
    \begin{tikzpicture}[>=Stealth,auto,semithick,initial text=,bend angle=12,
    font=\footnotesize,initial/.style=initial left]
      \node[state,initial] (i) {};
      \node[state] (a) at (6,3) {};
      \node[state] (c) at (6,-3) {};
      \node[state] (b) at (3.5,0) {};
      \path[->] (i)  edge [bend left] node {$a/a$} (a)
                edge [bend right] node [below left] {$c/c$} (c)
                edge node {$b/bb$} (b)
             (a)  edge [bend left] node {$c/c$} (c)
                 edge [loop right] node {$a/ba$} ()
                edge node {$b/bb$} (b)
             (b) edge [loop right] node {$b/b$} ()
                 edge [bend left] node {$a/a$} (a)
                 edge [bend right] node [below left] {$c/c$} (c)
             (c) edge node {$a/a$} (a)
                 edge node [above right] {$b/bb$} (b)
                 edge [loop right] node {$c/bc$} ();
    \end{tikzpicture}
    {
    \caption{Diagram of a transducer realizing the $\xi$ map. Labels show
    input/output.}
    \label{f:bseq}}
\end{figure}

In the following results, we will often need to
add an occurrence of $b$ at the end of $\xi(x)$, when $x$ ends in $c$.
This can be expressed as $\xi(xc)c^{-1}$, i.e.,
\begin{equation}
	\label{e:xi-}
	\xi(xc)c^{-1}=\begin{cases}\xi(x)b &\text{ if $x$ ends in $c$,}\\
	\xi(x) &\text{ otherwise.}\end{cases}
\end{equation}
Similarly, $c^{-1}\xi(cx)$ adds a $b$ before a possible leading $c$, and
$c^{-1}\xi(cxc)c^{-1}$ deals with both ends (but note that
$c^{-1}\xi(c\eps c)c^{-1}=b$).
\begin{rem}
	\label{t:image}
    The mapping on $\A^*$ given by $y\mapsto c^{-1}\xi(cyc)c^{-1}$ is
    injective. Its image is the regular language of words $x\in\A^+$
    satisfying all of the following conditions:
    \begin{enumerate}
    	\item $aa,abc,cba,cc\notin\F(x)$,
		\item $c$ is neither a prefix nor a suffix of $x$,
	    \item $x$ does not begin with $ba$ or end with $ab$.
    \end{enumerate}
    Hence, for each $x$ satisfying such conditions there exists a unique
    $y\in\A^*$ such that $x=c^{-1}\xi(cyc)c^{-1}$. This $y$ is obtained
    from $x$ by simply deleting an occurrence of $b$ from each run.
\end{rem}

\begin{lemma}
	\label{t:monob}
	Let $x,y\in\A^*$. The following are equivalent:
	\begin{enumerate}
		\item\label{i:x<y} $x<y$,
		\item\label{i:xixy} $\xi(x)<\xi(y)$,
		\item\label{i:xi-} $\xi(xc)c^{-1}<\xi(yc)c^{-1}$.
	\end{enumerate}
\end{lemma}
\begin{proof}
	Suppose $x<y$. If $y$ is a proper prefix of $x$, then $\xi(y)$ is a
	proper prefix of $\xi(x)$ since $\xi$ is sequential. Otherwise, let
	$x=ua'v$ and $y=ub'w$ for some $u,v,w\in\A^*$ and $a',b'\in\A$ with $a'<b'$.
	If $u$ ends in $b$ or is empty, then $\xi(x)$ and $\xi(y)$ begin
	respectively in $\xi(u)a'$ and $\xi(u)b'$; this also occurs if $u$ ends
	in $a$ and $a'\neq a$, or if $u$ ends in $c$ and $b'\neq c$. If we
	have $a'=a$ with $u$ ending in $a,$  then $\xi(x)$ begins with
	$\xi(u)ba$ while $\xi(y)$ begins with $\xi(u)bb$ or $\xi(u)c$;
	the case where $b'=c$ and $u$ ends in $c$ is similar. Thus
	$\cref{i:x<y}\Rightarrow\cref{i:xixy}$.
	
	Next, assume $\xi(x)<\xi(y)$. If $\xi(y)$ is not a prefix of
	$\xi(x)$, inequality~\ref{i:xi-} follows in view of
	\cref{e:xi-}; the same happens if $\xi(y)$ does not end in $c$, since
	in this case
	$\xi(xc)c^{-1}\leq\xi(x)<\xi(y)=\xi(yc)c^{-1}$.
	Let then $\xi(y)$ be a proper prefix of $\xi(x)$, with
	$\xi(y)$ ending in $c$. Now, $\xi(x)$ cannot begin with	$\xi(y)c$,
	since by definition $cc$ cannot occur in an image under	$\xi$. Hence
	$\xi(x)$ has either $\xi(y)a$ or $\xi(y)b$ as a prefix. In both cases we
	have $\xi(xc)c^{-1}\leq\xi(x)<\xi(y)b=\xi(yc)c^{-1}$ (note that
	$\xi(xc)c^{-1}\neq\xi(y)b$, since $\xi(x)\neq\xi(y)$). Hence
	$\cref{i:xixy}\Rightarrow\cref{i:xi-}$ in all cases.
	
	Finally, suppose inequality~\ref{i:xi-} from the statement holds. If
	$\xi(yc)c^{-1}$ is a proper prefix of $\xi(xc)c^{-1}$, then so is
	$\xi(y)$ in particular, and this implies that $y$ is a proper prefix of
	$x$. If	$\xi(xc)c^{-1}=Ua'V$ and $\xi(yc)c^{-1}=Ub'W$ for some
	$U,V,W\in\A^*$ and $a',b'\in\A$ with $a'<b'$, then there exist $u\in\A^*$
	such that either $U=\xi(u)b$ or $U=\xi(u)$. In the first case, it is
	easy to see that $x$ and $y$ respectively begin with $ua'$ and $ub'$, so
	that $x<y$. The same happens when $U=\xi(u)$ ends in $b$. If $U$
	ends in $c$, then $a'=a$ and $b'=b$ since $cc$ cannot appear
	in an image under $\xi$; hence $x$ begins with $ua$ while $y$ either
	equals $u$, or it begins with $ub$ or $uc$. The case where $U$ ends in
	$a$ is symmetric, except for the $x=u$ option which cannot occur. In all
	cases, we get $\cref{i:xi-}\Rightarrow\cref{i:x<y}$ as desired.
\end{proof}

We can now prove the analogue of \Cref{t:dAl,t:dCl} for the letter $b$.
\begin{lemma}
	\label{t:dBl}
	A word $x\in\A^*$ is singular if and only if $X=c^{-1}\xi(cxc)c^{-1}$ is
	singular.
\end{lemma}

\begin{proof}
    Assume $|X|_a=|X|_c+1$. By \Cref{t:delta0}, $X$ is singular if and only if
    either $X$ or $\overline{X}$ is $ab^n(ca)^m$ for some $m\geq 0$ and
    $n\neq 1$
    (as $abc,cba\notin\F(x)$); equivalently, $x$ or $\overline{x}$ is
    $ab^{n-1}(ca)^m$ (or $a(ca)^m$ when $n=0$), which happens if and only if
    $x$ is singular, as $|x|_a=|X|_a$ and $|x|_c=|X|_c$.
    
    Now suppose $|X|_a\neq|X|_c+1$. We may also assume that $X$ contains all three letters. Indeed, if $a$
    (resp.~$c$) does not occur, then $X=\lambda_b(x)b$ (resp.~$X=\rho_b(x)b^{-1}$) and the assertion
    follows from \Cref{t:dAl} (resp.~\Cref{t:dCl}). If
    $X$ does not contain $b,$ then $X$ is singular only if
    $X=a(ca)^m$ for some $m\geq 0$, which contradicts $|X|_a\neq|X|_c+1$.
    
    Let $X$ be singular. We need to show
    that no factorization $x=\overline{u}\cdot v\cdot w$ is reversible; without loss
    of generality, we may assume $v=a'v'b'$ for some letters $a'<b'$. If $u=\eps$,
    then $w\leq u$ and the factorization is not reversible, as desired; let
    then $u_1$ be the first letter of $u$. We have three options for $a',b'$:
    \begin{enumerate} 
        \item\label{i:ab} $a'=a$, $b'=b$.\\
        If $w=\eps$, then $X$ ends in $\xi(av'b)$; by \Cref{t:a1st},
        this is impossible if $u\neq\eps$.
        
        Let then $w$ begin with $w_1\in\A$. If $w_1<u_1$, then $w<u$ and we are
        done, so we only need to check the cases where $u_1\leq w_1$. By
        \Cref{t:bruns}, $\F(X)\cap\A^2$ is contained in $\{ab,ac,ba,bb,ca\}$
        if $|X|_a>|X|_c+1$, or in $\{ac,bb,bc,ca,cb\}$ if the opposite
        inequality holds.
        This excludes the sub-cases where $u_1\neq c$ and $w_1=c$.
        \begin{itemize}
            \item If $u_1=w_1=a$, then in $X$ a single $b$ is inserted between
            $u_1$ and $a'$. Since the factorization
            \begin{equation}
                \label{e:+0-+}
                X=c^{-1}\xi(c\overline{u})b\cdot\xi(v)b^{-1}
                \cdot b\xi(wc)c^{-1}
            \end{equation}
            cannot be reversible, we have $\xi(wc)c^{-1}\leq\xi(uc)c^{-1}$,
            so that $w\leq u$ by \Cref{t:monob}.
            \item If $u_1=a$ and $w_1=b$, then
            \begin{equation}
                \label{e:+0-0}
                X=c^{-1}\xi(c\overline{u})b\cdot\xi(v)b^{-1}
                \cdot\xi(wc)c^{-1}
            \end{equation}
            would always be reversible; hence this case cannot happen.
            \item If $u_1=w_1=b$, then as
            \begin{equation}
                \label{e:00-0}
                X=c^{-1}\xi(c\overline{u})\cdot\xi(v)b^{-1}
                \cdot\xi(wc)c^{-1}
            \end{equation}
            cannot be reversible, we again obtain $w\leq u$ by \Cref{t:monob}.
            \item Finally, if $u_1=w_1=c$, the non-reversibility of
            \begin{equation}
                \label{e:0000}
                X=c^{-1}\xi(c\overline{u})\cdot\xi(v)
                \cdot\xi(wc)c^{-1}
            \end{equation}
            again implies $w\leq u$.
        \end{itemize}
        \item $a'=a$, $b'=c$.\\
        If $w=\eps$, then $X$ ends in $\xi(av'c)b$; once again, this is
        impossible if $u\neq\eps$, by \Cref{t:a1st}. Let $w$ begin with
        $w_1\in\A$ such that $u_1\leq w_1$, again. As in case~\ref{i:ab}, the
        possible factors of length $2$ in $X$ exclude some options for $u_1$
        and $w_1$. We only need to examine the following.
        \begin{itemize}
            \item If $u_1=w_1=a$, then as
            \begin{equation}
                \label{e:0+00}
                X=c^{-1}\xi(c\overline{u})\cdot b\xi(v)\cdot\xi(wc)c^{-1}
            \end{equation}
            cannot be reversible, we obtain $w\leq u$ by \Cref{t:monob}.
            \item If $u_1=w_1=c$, then as
            \begin{equation}
                \label{e:00+0}
                X=c^{-1}\xi(c\overline{u})\cdot\xi(v)b\cdot\xi(wc)c^{-1}
            \end{equation}
            is not reversible, it follows that $w\leq u$.
        \end{itemize}
        \item $a'=b$, $b'=c$.\\
        If $w=\eps$, then $X$ ends in $V=\xi(bv'c)b$, which in turn begins
        with $bb$, so that $V<\overline{V}$. Since a factorization 
        $X=\overline{U}\cdot V\cdot\eps$ with $U\neq\eps$ (as $u\neq\eps$) is
        reversible, we obtain $w\neq\eps$ also in this case.
        If $w_1$ is its first letter and $u_1\leq w_1$, the following sub-cases
        are left to examine after factors of length $2$ in $X$ are considered.
        \begin{itemize}
            \item If $u_1=w_1=a$, then we obtain $w\leq u$ from the
            factorization in \cref{e:0000}.
            \item If $u_1=w_1=b$, then considering the factorization
            \begin{equation}
                \label{e:0-00}
                X=c^{-1}\xi(c\overline{u})\cdot b^{-1}\xi(v)\cdot\xi(wc)c^{-1}
            \end{equation}
            leads to the conclusion $w\leq u$ as above.
            \item If $u_1=b$ and $w_1=c$, then the factorization
            \begin{equation}
                \label{e:0-0+}
                X=c^{-1}\xi(c\overline{u})\cdot b^{-1}\xi(v)\cdot b\xi(wc)c^{-1}
            \end{equation}
            is always reversible; therefore this cannot occur.
            \item If $u_1=w_1=c$, then we get $w\leq u$ from
            \begin{equation}
                \label{e:+-0+}
                X=c^{-1}\xi(c\overline{u})b\cdot b^{-1}\xi(v)
                \cdot b\xi(wc)c^{-1}.
            \end{equation}
        \end{itemize}
    \end{enumerate}
    This concludes the proof that if $X$ is singular then so is $x$.
    
    Conversely, let us suppose that $x$ is singular with $|x|_a\neq |x|_c+1$,
    and prove the same for $X$. We need to show that no factorization
    $X=\overline{U}\cdot V\cdot W$ is reversible, where $V=a'V'b'$ for some $a',b'\in\A$
    with $a'<b'$, without loss of generality.
    
    If $U=\eps$, there is nothing to prove as $W\leq U$, so let $U_1$ be the
    first letter of $U$. If $W$ were empty, then since $X$ cannot end in $c$,
    we would have $a'=a$ and $b'=b$. Hence $x$ would not end in $a$; by
    \Cref{t:a1st}, $a'=a$ would then be a prefix of $x$, and then of $X$,
    against $U\neq\eps$.
    Let then $W$ begin with $W_1\in\A$, with $U_1\leq W_1$ (otherwise $W<U$ and
    there is nothing to prove).
    
    By \Cref{t:bruns}, the set $\F(x)\cap\A^2$ is contained in either
    $\{ac,bb,bc,ca,cb,cc\}$ or $\{aa,ab,ac,ba,bb,ca\}$, depending on the sign
    of $|x|_c-|x|_a+1$. This implies that $\F(X)\cap\A^2$ is contained in
    $\{ac,bb,bc,ca,cb\}$ or $\{ab,ac,ba,bb,ca\}$, respectively. Hence, the
    only cases left to examine are the following:
    \begin{enumerate}
        \item $U_1=a$, $a'=b$, $b'=c$, and $W_1=a$.\\
        Then $V$ begins with either $ba$ or $bb$. In the first case, there
        exist $u,v,w\in\A^+$ such that $x=\overline{u}vw$ and our factorization
        $X=\overline{U}VW$ coincides with \cref{e:0+00}; as $x$ is singular, we have
        $w\leq u$ and then $W\leq U$ by \Cref{t:monob}. In the second case,
        the same argument applied to \cref{e:0000} leads to $W\leq U$ as well. 
        \item $U_1=b$, $a'=a$, $b'=b$, and $W_1=b$.\\
        By \Cref{t:image}, $X$ cannot begin with $ba$, so that $U$ begins with
        $ba$ or $bb$. The same is true for $W$, for otherwise $x$ would end
        with $b$ and have an internal occurrence of $a$, against
        \Cref{t:a1st}.
        \begin{itemize}
            \item If $U$ and $W$ both begin with $ba$, there exist $u,v,w$
            such that $x=\overline{u}vw$, and \cref{e:+0-+} is
            $X=\overline{U}\cdot V\cdot W$.
            Then $w\leq u$, which implies $W\leq U$ by \Cref{t:monob}.
            \item If $U$ began with $ba$, but $W$ with $bb$, we would obtain
            \cref{e:+0-0} and $x=\overline{u}vw$ for some $u<w$, which is
            impossible.
            \item If $U$ begins with $bb$ and $W$ with $ba$, we have $W\leq U$
            as desired.
            \item Finally, if both $U$ and $W$ begin with $bb$, \cref{e:00-0}
            and $x=\overline{u}vw$ describe the situation; again $w\leq u$ follows,
            so that $W\leq U$.
        \end{itemize}
        \item $U_1=b$, $a'=b$, $b'=c$, and $W_1=b$.\\
        If $U=b$, then $W\leq U$ and we are done. If $U\neq b$ and $W=b$
        instead, $x$ would end with $c$, against \Cref{t:a1st}. Hence we can
        assume $U$ and $W$ begin in either $bb$ or $bc$.
        \begin{itemize}
            \item If $U$ and $W$ both begin with $bb$, \cref{e:0-00} applies
            and yields $W\leq U$.
            \item If $U$ began with $bb$ while $W$ began with $bc$,
            \cref{e:0-0+} would apply for some $u<w$, a contradiction.
            \item If $U$ begins with $bc$ and $W$ begins with $bb$, there is
            nothing to prove.
            \item If $U$ and $W$ both begin with $bc$, \cref{e:+-0+} describes
            the situation and implies $W\leq U$.
        \end{itemize}
        \item $U_1=c$, $a'=a$, $b'=b$, and $W_1=c$.\\
        Then $V$ ends in either $bb$ or $cb$, leading to \cref{e:0000,e:00+0}
        respectively, and then to $W\leq U$ in both cases.
    \end{enumerate}
    The proof is now complete.
\end{proof}

\noindent We now prove one of the main results of this section.
\begin{thm}
	\label{t:ternary}
    Let $\A=\{a,b,c\}$ with $a<b<c$. Every abelian class of $\A^*$ contains
    exactly one pair $\{x,\overline{x}\}$ of singular words.
\end{thm}
\begin{proof}
	Let $\pv=(n_a,n_b,n_c)\in\mathbb N^3$. We need to show that up to
	reversal, there exists a unique singular word $x$ with Parikh
	vector $\pv$, i.e., such that	$|x|_a=n_a$, $|x|_b=n_b$, and $|x|_c=n_c$.
	If $n_an_bn_c=0$, this follows from \Cref{t:Chr}, since there exists a
	unique power of a Christoffel word for every non-zero Parikh vector. In
	particular, this proves the assertion for $n_a+n_b+n_c=|x|\leq 2$.
	
	Let then $n_an_bn_c>0$, and assume that the result holds for all vectors
	$(p,q,r)$ with $p+q+r<n_a+n_b+n_c$. We can identify three cases:
	\begin{enumerate}
		\item If $n_a\geq n_b+n_c+1$, then by induction there exists a unique
		pair of singular words $\{x',\overline{x'}\}$ with Parikh vector
		$(n_a-n_b-n_c-1,n_b,n_c)$.
		Then $x=\lambda_a(x')a$ and $\overline{x}$ have vector $\pv$, and are
		singular by \Cref{t:dAl}. Any other singular word $y$ with vector
		$\pv$ begins and ends with $a$ by \Cref{t:a1st}, and has $a$ as a
		separating letter by \Cref{t:asep}; hence $y=\lambda_a(y')a$ for some
		$y'$ with vector $(n_a-n_b-n_c-1,n_b,n_c)$. Such a $y'$ is necessarily
		singular by \Cref{t:dAl}, so that $y'\in\{x',\overline{x'}\}$. This case is
		therefore settled.
		\item If $n_c\geq n_a+n_b-1$, then by induction there exists a unique
		pair of singular words $\{x',\overline{x'}\}$ with vector
		$(n_a,n_b,n_c-n_a-n_b+1)$ (note that $n_a+n_b-1\geq 1$ since $pq>0$).
		Then $x=\rho_c(x')c^{-1}$ and $\overline{x}$ have vector $\pv$, and are
		singular by \Cref{t:dCl}. By \Cref{t:csep}, any singular word $y$ with
		vector $\pv$ has $c$ as a separating letter, so that
		$y=\rho_c(y')c^{-1}$ for some word $y'$, which is singular by
		\Cref{t:dCl} and has vector $(n_a,n_b,n_c-n_a-n_b+1)$. Thus
		$y'\in\{x',\overline{x'}\}$, as desired.
		\item If $n_a<n_b+n_c+1$ and $n_c<n_a+n_b-1$, it follows
		$n_b>|\delta|$,
		with $\delta=n_c-n_a+1$. By \Cref{t:delta0}, if $\delta=0$ then
		$x=ab^{n_b}(ca)^{n_c}$ and $\overline{x}$ are the only singular words with
		vector $\pv$. Let us then assume $\delta\neq 0$.
		
		By induction, there exists a unique pair of singular words
		$\{x',\overline{x'}\}$ with vector $(n_a,n_b-|\delta|,n_c)$.
		Let $x=c^{-1}\xi(cx'c)c^{-1}$. Then $x$ and $\overline{x}$ are singular by
		\Cref{t:dBl}. Since $aa,cc\notin\F(x)$, by \Cref{t:asep,t:csep} we
		obtain $|x|_b\geq|\delta|$. \Cref{t:bruns} then implies that
		$x$ and $\overline{x}$ have exactly $|\delta|$ runs of $b$, and hence Parikh
		vector $\pv$.
		If $y$ is any singular word with vector $\pv$, then $aa,cc\notin\F(y)$
		by \Cref{t:asep,t:csep}. Furthermore, by \Cref{t:a1st}, $y$ cannot
		begins with $c$ or $ba$, and it cannot end with $c$ or $ab$.
		Finally, we have
		$abc,cba\notin\F(y)$ by \Cref{t:ab+bc}, as $\delta\neq0$.
		In view of \Cref{t:image}, this implies the existence of a unique word
		$y'$ such that $y=c^{-1}\xi(cy'c)c^{-1}$. By \Cref{t:bruns}, $y'$ has
		vector $(n_a,n_b-|\delta|,n_c)$, and by \Cref{t:dBl} it is singular.
		Hence we once again obtain $y'\in\{x',\overline{x'}\}$.\qedhere
	\end{enumerate}
\end{proof}

Behind induction, the above proof hides the following algorithm for
determining the unique (up to reversal) singular word with a given Parikh
vector $(n_a,n_b,n_c)$:
\begin{enumerate}
    \item\label{i:reduce} Starting from the vector $(n_a,n_b,n_c),$ we iterate the following rule:
    \begin{itemize}
        \item If $n_a\geq n_b+n_c+1$, the next vector is
        $(n_a-n_b-n_c-1,n_b,n_c)$;
        \item If $n_c\geq n_a+n_b-1$, the subsequent vector is $(n_a,n_b,n_c-n_a-n_b+1)$;
        \item Otherwise, the next vector is $(n_a,n_b-|n_c-n_a+1|,n_c)$.
    \end{itemize}
    \item Repeat step~\ref{i:reduce} until a vector $(p,q,r)$ is reached with $pqr=0$
    or $p=r+1$ is reached.
    \item Use \Cref{t:Chr} or \Cref{t:delta0}, respectively, to find the unique (up to reversal)
    singular word  with vector $(p,q,r)$.
    \item Go back through the previous vectors, using the appropriate result
    (\Cref{t:dAl}, \Cref{t:dCl}, or \Cref{t:dBl}) to construct the
    corresponding singular word.
\end{enumerate}

\begin{example}
Let us show that the only singular words with Parikh vector $(3,5,7)$ are
$x=acbcbcbcacbcbca$ and its reverse. Step~\ref{i:reduce} yields vector
$(3,5,0)$, and since the only power of a Christoffel word with $3-1$
occurrences of $a$ and $5+1$ of $b$ is $abbbabbb$, \Cref{t:Chr} implies that
$x'=abbbabba$ is singular. Hence $x=\rho_c(x')c^{-1}$ is singular too, by
\Cref{t:dCl}.

The singular words with vector $(3,7,5)$ are $y=acbbbcbbcacbbca$ and
$\overline{y}$.
Indeed, the above algorithm gives the sequence of vectors
$(3,4,5)\mapsto (3,1,5)\mapsto (3,1,2)$; by \Cref{t:delta0}, $y_0=abcaca$ is
singular, whence so are $y_1=\rho_c(y_0)c^{-1}=acbccacca$,
$y_2=c^{-1}\xi(cy_1c)c^{-1}=acbbcbcacbca$, and $y=c^{-1}\xi(cy_2c)c^{-1}$.
\end{example}

\noindent We observe that \Cref{t:ternary} is essentially a restatement of the \Cref{gen} given in the introduction. The next example shows that \Cref{t:ternary} does not extend to larger alphabets.

\begin{example}\label{counter} Consider the word $y=abbcdd$ on the ordered alphabet $\A=\{a,b,c,d\}$ with $a<b<c<d$. Then the abelian class of $y$ contains two singular words (up to reversal), namely $x=bcdbda$ and $x'=bdbcda$. Thus relative to any order preserving assignment $\phi: \{a,b,c,d\} \rightarrow \{2,3,4,\ldots \}$, the maximum of $\Ks(\cdot)$ will be attained at either $x$ or $x'$. Now, one can check that for the assignment $(a,b,c,d)\mapsto (3,4,5,6)$ one finds $\Ks(x)=6827>6825=\Ks(x')$ and hence amongst all permutations of $344566$ the maximum of $\Ks(\cdot)$ is attained uniquely at $x$ (or $\overline x)$. In contrast, if $(a,b,c,d)\mapsto (3,4,15,16)$ then $\Ks(x)=171127<171135 =\Ks(x')$, whence the maximum of $\Ks(\cdot)$ is attained uniquely at $x'$. This shows that in this case the maximizing arrangement depends on the actual values assigned to each of $a,b,c$ and $d$.  Finally, relative to the assignment $(a,b,c,d)\mapsto (3,4,7,8)$ one finds that $\Ks(x)=\Ks(x')=18247$ which shows that the maximizing arrangement need not be unique. 
\end{example}

\section{Infinite singular words}

Let $\A$ be an ordered set. For $x \in \A^\nats \cup \A^\ints$, let $L(x)$ denote the set of all factors $u\in \A^+$ occurring in $x$. We do not assume that each $a\in \A$ occurs in $x.$ We say $L(x)$ is \emph{symmetric} if  $L(x)$ is closed under reversal, i.e., if $u\in L(x)$ then $\overline u \in L(x)$.

\noindent We begin by investigating infinite binary singular words.

\begin{prop}\label{fullbal}Let $x\in \A^\ints$ be a bi-infinite word over an ordered  binary alphabet $\A=\{a,b\}$. Then $x$ is singular if and only if $x$ is balanced.
\end{prop}

\begin{proof} By \Cref{bal}, if $x$ is singular then $x$ is balanced. Conversely, if $x$ is not singular, then $x$ admits a reversible factorization $x=\overline u v w$. Thus we may write $v=ra'v'b'\overline r$, $u=sa'u'$ and $w=sb'w'$ with $r,s,v' \in \A^*$, $u',w' \in \A^\nats$ and $\{a',b'\}=\{a,b\}$. It follows that $a'\overline s r a'$ and $b'\overline r s b'$ are each factors of $x$, whence $x$ is not balanced.  \end{proof}

\noindent As an immediate corollary:

\begin{cor}\label{MP} Let $x\in \A^\ints$ be a bi-infinite word over a binary alphabet $\A=\{a,b\}$. Then $x$ verifies the Markoff property (M) if and only if $x$ is singular for every linear order $<$ on $\A$. 
\end{cor}

\begin{proof} Fix a linear order $<$ on $\A$. Then by \Cref{fullbal},  $x$ is singular if and only if $x$ is balanced. The result now follows from 
Proposition~3.1 in \cite{Reu}. \end{proof}

In order to characterize one sided infinite binary singular words, we first establish the following general lemma: 

\begin{lemma}\label{Lyndon} Let $\A$ be any ordered alphabet.  Let $x\in \A^\nats $ and assume $L(x)$ is symmetric. If $x$ is singular, then $x\leq x'$ for each suffix $x'$ of $x$. In particular, if $x$ is not (purely) periodic, then $x$ is Lyndon.
\end{lemma}

\begin{proof} Assume to the contrary that $x' < x$ for some proper suffix $x'$ of $x$. Then there exist $z\in \A^*$, $a,b\in \A$ with $a<b$ such that $zb$ is a prefix of $x$ and $za$ a prefix of $x'$. Let $v$ be a prefix of $x$ beginning in $zb$ and ending in $a\overline z$. Then we have $x=vw$ for some $w\in \A^\nats$ and $\overline v < v$. It follows that the factorization $x=\overline u v w$ with $u=\varepsilon$ is reversible, whence $x$ is not singular, a contradiction. Finally, if $x$ is not purely periodic, then $x\neq x'$ for each proper suffix $x'$ of $x.$ It follows that $x<x'$ for each proper suffix $x'$ of $x$ and hence $x$ is infinite Lyndon.  \end{proof}

\LyndSt*

\begin{proof}Without loss of generality we may assume that $a<b$. First assume $x$ is singular. Then $x$ is balanced by \Cref{bal}.
It follows that $x$ is Sturmian (see Theorem~2.1.5 in \cite{Lot}). Since $L(x)$ is symmetric (see Proposition~2.1.19 in \cite{Lot}), it follows from \Cref{Lyndon} that $x$ is Lyndon.  

For the converse, assume $x$ is a Sturmian Lyndon word, and suppose to the contrary that $x$ admits a reversible factorization $x=\overline u vw$. Of all such factorizations, pick one with the length of $v$ minimal. It follows that the first and last letter of $v$ are distinct. 
Thus, assume $v$ begins in $c$ and ends in $d$ with $\{c,d\}=\{a,b\}$.   If $u$ is a prefix of $w$, then $w<u$ and hence $\overline v <v$ which implies that $c=b$ and $d=a$. If follows that $au$ is a factor of $x$ and hence so is $\overline ua$. But $x$ begins in $\overline ub$ which contradicts the fact that $x$ is Lyndon.   Thus $u$ is not a prefix of $w$.  In this case we have  $u=zeu'$ and $w=zfw'$ with $z,u' \in \A^*$,  $w'\in \A^\nats$ and $e$ and $f$ distinct letters. In particular $z$ is a right special factor of $x$.
But since $\overline zc$ and $\overline zd$ are also factors of $x$, it follows that $\overline z$ is also right special, whence $z=\overline z$. 
Thus $ezc$ and $dzf$ are each factors of $x$. If $v<\overline v$ (i.e., $c<d)$ then $u<w$ (i.e, $e<f)$ whence $a=c=e$ and $b=d=f$ from which it follows that $aza$ and $bzb$ are each factors of $x$, a contradiction. Similarly, if $v>\overline v$ then $u>w$ whence $b=c=e$ and $a=d=f$ which again leads to the contradiction that $aza$ and $bzb$ are both factors of $x$.  This concludes our proof of \Cref{Sturmian}. \end{proof}

\begin{cor}\label{UR}Let $x\in \A^\nats$ be an aperiodic word over an ordered  binary alphabet $\A=\{a,b\}$. If $x$ is singular then $x$ is uniformly recurrent.
\end{cor} 
\begin{proof} This follows immediately from \Cref{Sturmian} together with the fact that every Sturmian word is uniformly recurrent (see for instance Proposition~2.1.25 in \cite{Lot}). \end{proof}

\begin{cor} Let $x\in \A^\nats$ be a uniformly recurrent aperiodic word over an ordered  binary alphabet $\A=\{a,b\}$. Then the shift orbit closure of $x$ contains a singular word if and only if $x$ is Sturmian. Furthermore, this singular word is unique.
\end{cor}

\begin{proof} If $x$ is Sturmian, then the unique Lyndon word in the shift orbit closure of $x$ is singular by \Cref{Sturmian}. Conversely, assume $y$ is a singular word in the shift orbit closure of a uniformly recurrent aperiodic binary word $x.$ Then $y$ is also uniformly recurrent, aperiodic and binary, whence $y$ is Sturmian Lyndon by \Cref{Sturmian}. It follows that $x$ is also Sturmian.  \end{proof}

The binary case is quite special as already on a ternary alphabet infinite singular words display a much different behavior. As we saw an infinite singular aperiodic word over a binary alphabet is necessarily uniformly recurrent and its set of factors is closed under reversal. Furthermore, a uniformly recurrent word  $x$ over a binary alphabet contains a singular word in its shift orbit closure if and only if $x$ is Sturmian, and moreover, this singular word is unique. This is in general not the case on alphabets of cardinality greater than $2$. For example, it is easy to see that over the ternary alphabet $\A=\{a,b,c\}$ with $a<b<c$, any infinite concatenation of $ac$ and $abc$ is singular. 
This implies in particular that there exist non recurrent singular words or recurrent aperiodic singular words which are not uniformly recurrent. It also shows that in general the set of factors of an infinite  singular word need not be closed under reversal. 

We next investigate the structure of  infinite singular words whose set of factors is closed under reversal. In what follows, let $\A_k=\{1,2,\ldots ,k\}$ be ordered by $1<2<\cdots <k$. 

\begin{defin}\label{order}Let $L\subseteq \A_k^*$. We say that $L$ satisfies the \emph{symmetric order condition} if whenever 
\begin{equation}\label{*}\mbox{$asd, bsc \in L$ with $a\neq b$ and $c\neq d$ $(a,b,c,d\in \A_k),$ then $a<b$ $\Leftrightarrow$ $c< d$}\end{equation}
for each $s\in \A_k^*$. 
\end{defin}

 \begin{thm}\label{H4}Let $x\in  \A_k^\nats \cup \A_k^\ints$. Assume $L(x)$ is symmetric. 
\begin{enumerate}
\item If $x\in \A_k^\ints, $ then $x$ is singular if and only if  $L(x)$ satisfies the symmetric order condition.
\item  If $x\in \A_k^\nats$,  then $x$ is singular if and only if  $L(x)$ satisfies the symmetric order condition and $x\leq x'$ for each suffix $x'$ of $x$. 
\end{enumerate}
\end{thm}

\begin{proof} We begin by showing that if $x\in  \A_k^\nats \cup \A_k^\ints$ is singular, then $L(x)$ verifies the symmetric order condition. So assume that $asd,bsc\in L(x)$ with $a\neq b$ and $c\neq d$. Since $L(x)$ is symmetric both $d\overline s a, bsc\in L(x)$. Furthermore, these two words cannot overlap one another. Thus there exists $r\in \A_k^*$ such that $d\overline s ar bsc \in L(x)$ or $bsc r d\overline sa \in L(x)$. If $d\overline s ar bsc \in L(x)$, then we can write $x=\overline u vw$ with $v=arb$, and where $w$ begins in $sc$ and $u$ begins in $sd$. Then $a<b \Leftrightarrow v<\overline v \Leftrightarrow w<u \Leftrightarrow c<d$. Similarly, if  $bsc r d\overline sa \in L(x)$ then we can write $x=\overline u vw$ with $v=crd$, and where $w$ begins in $\overline s a$ and $u$ begins in $\overline s b$. Then
$c<d \Leftrightarrow  v< \overline v \Leftrightarrow w<u \Leftrightarrow a<b$. Also, if $x\in \A_k^\nats$, then by \Cref{Lyndon} we have $x\leq x'$ for each suffix $x'$ of $x.$

For the converse, first assume that $x\in \A_k^\ints$. We will show that if  $x=\overline u v w$ with $v\in \A_k^+$  and $v\neq \overline v$ and $u\neq w$ then $v<\overline v$ if and only if $w<u$. Without loss of generality we may assume that $v$ begins and ends in distinct letters $c$ and $d$. Thus fix a factorization $x=\overline u v w$ with $u\neq w$ and with  $v=cyd$ with $y\in \A_k^*$ and $c,d\in \A_k$ distinct. 
Let us write $u=sbu'$ and $w=saw'$ with $s\in \A_k^*$,  $u', w'\in \A_k^\nats$ and $a,b\in \A_k$ distinct. Since $L(x)$ is symmetric, both $a\overline s d$ and $b\overline s c$ belong to $L(x)$. It follows from the symmetric order condition that if $a<b \Leftrightarrow c<d$.  In other words, if $w<u \Leftrightarrow v<\overline v$  as required. 

Next assume that $x\in \A_k^\nats$ and $L(x)$ satisfies the symmetric order condition and $x\leq x'$ for every suffix $x'$ of $x$.  Fix a factorization $x=\overline u v w$ with $v=cyd$ with $y\in \A_k^*$ and $c,d\in \A_k$ distinct and $u\neq w$. Then, the same proof as in the case of $x\in \A_k^\ints$ shows that if $u$ is not a proper prefix of $w$, then $v<\overline v \Leftrightarrow w<u$. So it remains to consider the case when $u$ is a proper prefix of $w$ in which case $w<u$. As $L(x)$ is symmetric and $du\in L(x)$ it follows that $\overline ud$ is a factor of $x$. Let $x'$ be a (proper) suffix of $x$ beginning in $\overline ud$. Since $\overline uc$ is a prefix of $x$ and $x\leq x'$,  it follows that $c<d$ and hence $v<\overline v$ as required.  \end{proof}

We will now use \Cref{H4} to show that infinite singular words $x\in \A_k^\ints$, with $L(x)$ symmetric, arise as natural codings of symmetric $k$-interval exchange transformations. In fact, the symmetric order condition is precisely the combinatorial criterion which distinguishes the language of symmetric interval exchange transformations from other symmetric languages of the same factor complexity including Arnoux-Rauzy subshifts or more generally Episturmian subshifts.  

Interval exchange transformations were originally introduced by Oseledec \cite{ose}, following an idea of Arnold \cite{arn}, see also \cite{ks}. A $k$-\emph{interval exchange transformation} $\mathcal I$ is given by a probability vector $(\alpha _1,\alpha _2,\ldots ,\alpha _k)$, $0<\alpha_i<1$, together with 
a permutation $\pi$ of $\{1,2,\ldots ,k\}$. The unit interval $[0,1]$ is partitioned into $k$ sub-intervals of lengths $\alpha _1, \alpha _2,\ldots ,\alpha _k$ labeled $1,2,\ldots ,k$ which are  then rearranged according to the permutation  $\pi^{-1}$. (Note: In some definitions, the intervals are rearranged according
to the permutation $\pi.$)
We will only be interested here in symmetric interval exchanges, that is,  where the permutation $\pi$ is given by $\pi(i)=k+1-i$ :

\begin{defin} A \emph{symmetric $k$-interval exchange transformation} $\mathcal I$ is a $k$-interval exchange transformation  with probability vector $(\alpha _1,\ldots,\alpha _k)$, and permutation $\sigma (i)=k+1-i$, $1\leq i\leq k$ defined by
\[\mathcal {I}(x)=x+\sum_{\pi^{-1}j<\pi^{-1}i}\alpha_{j}-\sum_{j<i}\alpha_{j}\]
when $x$ belongs to the (half open) interval
\[\mathcal I_{i}=\left[ \sum_{j<i}\alpha_{j}
,\sum_{j\leq i}\alpha_{j}\right).\]
We denote by $\beta_{i}$ $(1\leq i\leq k-1)$, the $i$-th point of discontinuity of $\mathcal {I}^{-1}$, namely 
$\beta_{i}=\sum_{j=k+1-i}^k\alpha_{j}$  and by 
$\gamma_{i}$ is the $i$-th discontinuity of $\mathcal {I}$, namely 
$\gamma_{i}=\sum_{j=1}^i\alpha_{j}$. Then
  $\mathcal I_{1}=[0, \gamma_{1})$, $\mathcal I_{i}=[\gamma_{i-1}, \gamma_{i})$, $2\leq i\leq k-1$
  and $\mathcal I_{k}=[\gamma_{k-1}, 1)$.
 \end{defin}
 
Two points $x,y \in [0,1]$ are said to belong to the same $\mathcal I$\emph{-orbit} if $\mathcal I^n(x)=y$ for some $n\in \ints$. This defines an equivalence relation on $[0,1]$ and the equivalence classes are called {\it orbits}. 
To each point $\gamma \in [0,1]$, one associates a bi-infinite word $(x_n)_{n\in \ints}\in \A_k^\ints$, called the \emph{natural coding} of $\gamma$ under $\mathcal I$,  where $x_n=i$ whenever $\mathcal{I}^n(\gamma) \in \mathcal I_i$. We define the language of $\mathcal I$, denoted $L(\mathcal I)$, to be the language generated by all natural codings, i.e., $w\in L(\mathcal I)$ if and only if $w$ is a factor of the  natural coding of some point $\gamma$ under $\mathcal I$.  

Via a decomposition result due to D.~Gaboriau, G.~Levitt and F.~Paulin which applies generally to all systems of partial isometries (see Theorem~3.1 in \cite{GLP}), every symmetric $k$-interval exchange transformation $\mathcal I$  decomposes canonically into a finite number of invariant sub-systems $\mathcal J_i.$ Moreover on each $\mathcal J_i,$ either every orbit is finite, meaning each corresponding natural coding is periodic, or each orbit is dense, in which case $\mathcal J_i$ is said to be {\it minimal}. In particular,  the closure of an orbit is not a Cantor set and  every natural coding under $\mathcal I$ of a point $\gamma\in [0,1]$  is uniformly recurrent. 

We will make use of the following recent result due to S.~Ferenczi, P.~Hubert and the third author \cite{FHZ} characterizing symmetric $k$-interval exchange transformations languages (see the proof of Theorem 13 and Lemma 11 in \cite{FHZ} or Proposition 4 in \cite{FZ}): 

\begin{thm}\label{thm1}[Theorem 13 in \cite{FHZ})
A language  $L\subseteq \A_k^*$  is the language of a symmetric $k$-interval exchange transformation $\mathcal I$ with interval lengths $(\alpha_1,\alpha_2,\ldots ,\alpha_k)$ if and only if the following two conditions are satisfied: 
\begin{itemize}
\item $L$ satisfies the symmetric order condition
\item {measure condition}: there exists an invariant probability measure $\mu$ on the symbolic dynamical system $(X_L,S)$ generated by $L$ such that
$\mu ([w])>0$
for each $w\in L$ and $\mu([i])=\alpha_i$ for each $i\in \A_k\cap L.$ \end{itemize}\end{thm}

 The language $L$ in \Cref{thm1} is assumed to be factorial and extendable. That is, writing $L=\bigcup_{i\geq 0} L_i$ with $L_0=\{\varepsilon\}$ and $L_n\subseteq\A_k^n$ for all $n$, we have that for each each $v\in L_n$ there exists $a,b\in\A_k$ with $av,vb\in L_{n+1}$ and each $v\in L_{n+1}$ may be written as $v=au=u'b$ with $a,b\in\A_k$ and  $u,u'\in L_n$.  A language $L$ is \emph{minimal} if for each $v\in L$ there exists $n$ such that $v$
is a factor of each word $w\in L_n$. The \emph{symbolic dynamical system} $(X_L,S)$ generated by a language $L$ is the two-sided shift  $S:X_L\rightarrow X_L$ where $X_L$ consists of all bi-infinite words $x\in \A_k^\ints$ such that $L(x)\subseteq L$ and where the shift operator $S$ is defined by $S((x_i)_{i\in \ints})_n=x_{n+1}$. Finally, by $[w]$ we mean the \emph{cylinder set} defined by $w \in L$, i.e., $[w]=\{x\in X_L: x_0x_1\cdots x_{|w|-1}=w\}$.

\begin{prop}\label{soc} Let $x\in \A_k^\ints$ be uniformly recurrent and assume that each $i\in \A_k$ occurs in $x.$ If $L(x)$ satisfies the symmetric order condition, then $L(x)$ is symmetric and hence $x$ is singular. 
\end{prop}

\begin{proof} Since $x$ is uniformly recurrent, it follows that $L(x)$ is minimal. Minimality of $L(x)$ in turn implies the measure condition given in \Cref{thm1}. In fact, as in \cite{Bos}, let $\mu$ be any invariant probability measure on the shift orbit closure of $x$, which is the same as the symbolic system  generated by  $L(x)$.  Then for each positive integer  $n$ there is at least one word $v(n)\in L(x)$ of length $n$ with $\mu ([v(n)])>0$.
Now for each $w\in L(x)$, as $x$ is uniformly recurrent, it follows that $w$ is a factor of $v(n)$ for all $n$ sufficiently  large. Hence  $\mu([w])\geq \mu([v(n)])>0$ for all $n$ sufficiently large.
It follows from \Cref{thm1} (applied to $L(x))$ that $L(x)=L(\mathcal I)$ for some minimal symmetric $k$-interval exchange transformation $\mathcal I$ with interval lengths $(\mu([1]), \mu([2]), \ldots ,\mu([k]))$. 

Now consider $\overline x\in  \A_k^\ints$. Then $L(\overline x)=\{w \in \A_k^*: \overline w \in L(x)\}$. 
Then $L(\overline x)$ also verifies the symmetric order condition. Furthermore, we can define an invariant probability measure $\mu'$ on the shift orbit closure of $\overline x$ by $\mu'([w])=\mu([\overline w])$ for each $w\in L(\overline x)$. It follows  that $L(\overline x)$ also satisfies the measure condition and hence by \Cref{thm1},  $L(\overline x)=L(\mathcal I')$ where $\mathcal I'$ is a minimal symmetric $k$-interval exchange transformation with interval lengths $(\mu'([1], \mu'([2]), \ldots ,\mu'([k])$. Since $\mu([i])=\mu'([i])$ for each $i\in \A_k$, it follows that $\mathcal I=\mathcal I'$ and hence $L(x)=L(\mathcal I) =L(\mathcal I')=L(\overline x)$. This proves that $L(x)$ is symmetric. Finally, that $x$ is singular now follows from \Cref{H4}. \end{proof} 

\symmiets*
\begin{proof}
Let $x\in \A_k^\ints$ and assume $L(x)=L(\mathcal I)$ for some  symmetric $k$-interval exchange transformation $\mathcal I$. We will show that $x$ is singular and $L(x)$ symmetric.  By \Cref{thm1}, $L(x)$ satisfies the symmetric order condition.  As $x$ is uniformly recurrent, it follows from \Cref{soc} that $L(x)$ is symmetric and $x$ singular.  

For the converse, let $x\in \A_k^\ints$ be uniformly recurrent. Assume $x$ is singular and $L(x)$ is symmetric. It follows from \Cref{H4} that $L(x)$ satisfies the symmetric order condition. Also, as $x$ is uniformly recurrent,  it follows (as in the proof of \Cref{soc}) that $L(x)$ satisfies the measure condition. Hence by \Cref{thm1}, we have that $L(x)=L(\mathcal I)$ for some minimal $k$-interval exchange transformation $\mathcal I$. \end{proof} 

We now give a characterization of natural codings of symmetric $k$-interval exchange transformations satisfying Keane's  infinite distinct orbit condition \cite{kea}:

\begin{defin} A $k$-interval exchange transformations $\mathcal I$  satisfies the \emph{infinite distinct orbit condition} (or \emph{i.d.o.c.} for short)  if the $k-1$ negative trajectories $\{{\mathcal I}^{-n}(\gamma_{i} )\}_{n\geq 0}$ $(1\leq  i\leq k-1)$ of the discontinuities of ${\mathcal I}$
are infinite disjoint sets.\end{defin}\medskip

If $\mathcal I$ satisfies  i.d.o.c., then  $\mathcal I$ is minimal but not conversely.  A complete characterization of languages of interval exchange transformations satisfying i.d.o.c.  was obtained by S. Ferenczi and the third author in \cite{FZ} (see \Cref{iets} below). It is based on  Kerckhoff's definition of $k$-interval exchange transformations which involves two permutations $(\pi_0,\pi_1)$ in which the unit interval $[0,1]$ is partitioned into $k$ sub-intervals of lengths $\alpha _1, \alpha _2,\ldots ,\alpha _k$ ordered according to the permutation  $\pi_0^{-1}$ and then rearranged according to the permutation  $\pi_1^{-1}$ (see \cite{ker}). We first need to recall some terminology:

\begin{defin}\label{dtg}
\noindent For a permutation $\pi$ of $\{1,2,\ldots, k\}$, we define the {\em $\pi$-order} by $a<_{\pi} b$ whenever $\pi (a)< \pi(b)$.
\noindent A {\em $\pi$-interval} is a nonempty set of consecutive integers in the $\pi$-order.  \end{defin}

Let $L$ be a language. For $w\in L$, we define {\em arrival} set of $w$, denoted $A(w)$, as the set of all letters $a$ such that $aw$ is in $L$, and the {\em departure} set of $w$, denoted  $D(w)$, as the set of all letters $a$ such that $wa\in L$.

The following theorem gives a characterization of languages generated by a $k$-interval exchange transformation satisfying  i.d.o.c.:

\begin{thm}\label{iets}[Theorem 2 in \cite{FZ}]
A language $L$ is the language of a $k$-interval exchange transformation $\mathcal I$, defined by permutations
$(\pi_0, \pi_1)$ such that $\pi_0^{-1}(\{1,... j\}) \neq \pi_1^{-1}(\{1,... j\})$ for every $1\leq j\leq k-1$, and satisfying  i.d.o.c., if and only if $L$ satisfies
\begin{enumerate}
\item[(H0)] $L_1=\{1,\ldots, k\}$,
\item[(H1)] $L$ is minimal,
\item[(H2)] if $w$ is a bispecial word, $A(w)$ is a $\pi_1$-interval, 
\item[(H3)] if $w$ is a bispecial word and $a\in A(w)$,  $D(aw)$ is a $\pi_0$-interval,
\item[(H4)] if $a,b\in A(w)$ with $a<_{\pi_1}b$, $c\in D(aw)$, $d\in D(bw)$, 
then $c\leq_{\pi_0}d$,
\item[(H5)] if $a,b\in A(w)$  are consecutive in the $\pi_1$ order, $D(aw)\cap D(bw)$ is a singleton. \end{enumerate}\end{thm}

\noindent We will also be needing the following lemma: 

 \begin{lemma}\label{sym}[Lemma 6 in \cite{FZ}] If $L$ satisfies $(H0)$ to $(H5)$ for $\pi_0=Id$, $\pi_1=\sigma \,: i\mapsto k+1-i$, then $L$ is symmetric. \end{lemma}
 
 As observed in \cite{FZ}, condition (H5) precludes the existence of weak bispecial factors (see \cite{Cass}). It is implied by i.d.o.c.. 
 In general, a natural coding $x$ of a symmetric $k$-interval exchange transformation may contain weak bispecial factors as the following example illustrates:

\begin{example} Let $y=01001010010\cdots$ be the Fibonacci word fixed by the morphism $0\mapsto 01$, $1\mapsto 0$. Let $x\in \{1,2,3\}^\nats$ be the image of $0y$ under the morphism $0\mapsto 1213$, $1\mapsto 12213$. 

\[x=12131213122131213121312213121312213121312131221312131\cdots.\]
Then it is readily verified that $x$ is an aperiodic uniformly recurrent singular word and $L(x)$ is symmetric (in fact, $x$ begins in infinitely many palindromes). It follows that $L(x)$ satisfies the symmetric order condition and the measure condition, whence $L(x)=L(\mathcal I)$ for some symmetric $3$-interval exchange transformation $\mathcal I.$  However, $\mathcal I$ does not satisfy i.d.o.c. since $L(x)$ does not verify (H5). In fact, consider the bispecial factor $w=1$. Then $2,3\in A(1)$ are consecutive in the $\pi_1$-order. However, $D(21)=\{3\}$ while $D(31)=\{2\}$ whence $D(21)\cap D(31)=\emptyset$.  
 \end{example}

 
 \begin{thm}\label{idoc} Let  $x\in \A_k^\ints$ be  uniformly recurrent and assume each $i\in \A_k$ occurs in $x.$  Assume that for each $w\in L(x)$ there exists $a\in \A$ such that $D(aw)=D(w)$. Then the following are equivalent
\begin{enumerate}
\item $x$ is singular and $L(x)$ is symmetric.
\item $L(x)$ is the language of a symmetric $k$-interval exchange transformation satisfying  i.d.o.c..
\end{enumerate}
\end{thm}

\begin{proof} That 2. implies 1. follows from \Cref{ietssymm}. To see that 1. implies 2. we begin with the following lemma:

\begin{lemma}\label{arrdep} Let $x\in \A_k^\ints$. Assume that $L(x)$ is symmetric and that each $i\in \A_k$ is recurrent in $x$. If $x$ is singular, then $A(s)$ and $D(s)$ are both intervals for each $s\in L(x)$.
\end{lemma} 

\begin{proof} Fix $s\in L(x)$. It suffices to show that $D(s)$ is an interval. In fact, since  $L(x)$ is symmetric it follows that $A(s)=D(\overline s)$ for each $s\in L(x)$. So assume that there exist $a<b<c$ in $\A_k$ with $a,c\in D(s)$. We will show that  $b\in D(s)$. Since $b$ is recurrent in $x$ we may write $x=\overline ubw$ with $u, w\in \A_k^\nats$ and with $sa$ and $sc$ each occurring in $\overline u.$ 
Thus we may write $x=\overline {u_1} v_1 w$ where $v_1$ begins in $a$ and ends in $b$ and $u_1$ begins in $\overline s$. Since $v_1<\overline {v_1}$ and $x$ is singular, it follows that $w\leq u_1$. Similarly, we may write  $x=\overline{u_2} v_2 w$ where $v_2$ begins in $c$ and ends in $b$ and $u_2$ begins in $\overline s$. Since $\overline {v_2} <v_2$ we have that $u_2\leq w$. Since $\overline s$ is a prefix of both $u_1$ and $u_2$ it follows that $\overline s$ is a prefix of $w$, whence $b\overline s \in L(x)$ from which it follows that $sb\in L(x)$ as required.%
\end{proof}

We now prove that 1. implies 2. So assume that $x\in \A_k^\ints$ satisfies the hypotheses of \Cref{idoc},   $x$ is singular and that $L(x)$ is symmetric. We  now show that $L(x)$ verifies each of the conditions in \Cref{iets}. Condition (H0) is immediate since each $i\in \A_k$ occurs in $x$. Condition (H1) follows from the fact that $x$ is uniformly recurrent. Conditions (H2) and (H3) follow immediately from \Cref{arrdep}. Condition  (H4) applied to $\pi_0=Id$ and $\pi_1=\sigma$ is merely a reformulation of the symmetric order condition. Thus (H4) follows from \Cref{H4}.
To show (H5), suppose $a,b\in A(w)$ are consecutive. Without loss of generality, we may assume that $a<b$.  Let $c,d \in D(aw)\cap D(bw)$. We will show that $c=d$. We have $awc, awd, d\overline w b, c\overline w b\in L(x)$. Considering the factors $awc$ and $d\overline w b$, since $a<b$ it follows that $wc\geq wd$ and hence $c\geq d$. Similarly, considering the factors $awd$ and $c\overline w b$, since $a<b$ we get $wd\geq wc$ and hence $d\geq c$. This proves that $c=d$ and hence that $D(aw)\cap D(bw)$ is at most a singleton. To see that $D(aw)\cap D(bw)\neq \emptyset$, pick $c\in \A_k$ such that $D(c\overline w)=D(\overline w)$. As $a,b\in D(\overline w)$ we get that $c\overline w a, c\overline wb \in L(x)$. It follows that $awc,bwc\in L(x)$ and hence $c\in D(aw)\cap D(bw)$. This concludes our proof of \Cref{idoc}.\end{proof}



\section*{Acknowledgments}
In memory of the late Gerhard Ramharter whose earlier work served as our main source of inspiration, and to whom we are extremely grateful for his invaluable input and many suggestions during the time of our research.

\end{document}